%% file: main.tex
\documentclass{article}
\usepackage[utf8]{inputenc}

\usepackage{amssymb,amsmath,amsthm}
\usepackage{enumerate}
\usepackage{tikz}
\usepackage{tikz-cd}
\usetikzlibrary{automata,positioning}
\usepackage{mathtools}
\usepackage{float}
\usepackage{graphicx}
\usepackage{dirtytalk}
\usepackage{chngcntr}
\usepackage{enumerate}
\usepackage{bbm}

\usepackage{amssymb}
\usepackage{graphicx}
\usepackage[hidelinks]{hyperref}
\usepackage{hyperref}
\usepackage{fancyhdr}
\newtheorem{theorem}{Theorem}[section]

\newtheorem{lemma}{Lemma}[section]

\theoremstyle{definition}
\newtheorem{definition}{Definition}[section]
\newtheorem{proposition}{Proposition}[section]
\newtheorem{example}{Example}[section]

\newtheorem*{remark}{Remark}

\usepackage{sectsty}
\sectionfont{\centering}
\usepackage[backend=biber,style=numeric-comp,]{biblatex}
\title{\textbf{A Beginner’s Guide to Homological Algebra: A Comprehensive Introduction for Students}}

\author{
  Andy Eskenazi\textsuperscript{1,2}\\
  \texttt{andyeske@sas.upenn.edu}
  \and
  Kevin You\textsuperscript{1} \\
  \texttt{kevinyou@sas.upenn.edu}
  \and
  Will Vauclain\textsuperscript{1,3}\\
  \texttt{willvauc@seas.upenn.edu}
  \and
  Robin Murugadoss\textsuperscript{1}\\
  \texttt{gmdoss@sas.upenn.edu}
}

\makeatletter
\renewcommand\@date{{
 \textsuperscript{1}\normalsize Department of Mathematics, University of Pennsylvania\par
 \textsuperscript{2}\normalsize Department of Mechanical Engineering and Applied Mechanics, University of Pennsylvania\par
 \textsuperscript{3}\normalsize Department of Computer and Information Science, University of Pennsylvania
 }}
\makeatother
\input{prelude.tex}

\begin{document}

\maketitle

\bigskip
\textbf{Abstract.} Homological algebra is often understood as the translator between the world of topology and algebra. However, this branch of mathematics is worth studying by itself, given that it provides fascinating perspectives about other disciplines, most notably, category theory. In this paper, we seek to provide an introductory guide for advanced students of mathematics and related specialties seeking to get started on homological algebra, covering the necessary central topics to later delve deeper into more complex aspects of this field and beyond. This work starts by presenting the notion of chain complexes of algebraic structures, and then moves into exploring homology modules and chain homotopies. Next, we provide an overview of projective resolutions and conclude by entering the world of category theory by looking at Tor functors.

\section{Introduction}
\label{sec:intro}

Homological algebra is a very fascinating branch of mathematics, one that finds its origins, generally speaking, in algebraic topology, in an attempt to construct bridges between topological shapes and algebraic structures. Indeed, since its inception, homological algebra has become a very valuable tool for topologists, providing an algebraic approach for the classification of various types of shapes. However, beyond the geometrical applications, this field is worth studying independently, due to its wonderful insights and the mechanism it supplies to view and understand complex concepts in the world of mathematics, such as category theory\footnote{Category Theory is a wonderful branch of mathematics that can be understood, vaguely speaking, as an abstraction from mathematical structures. In essence, this field deals with mathematical objects, such as sets, groups, topologies, and the relationships between them. In that sense, categories are composed of objects such as those listed previously, as well as the relationships (or morphisms) between them. For instance, the category of sets contains all sets as its objects and functions as its morphisms, while the category of groups contains groups as its objects and group homomorphisms as its morphisms. This work will often times point out that an algebraic structure is part of a category to make special emphasis on the properties it might posses. The definition of a category will be provided more rigorously in Section 2 of this work}, in a different and simpler way.

Through this work, we hope to provide the advanced undergraduate or first year math graduate student with a complete, self-contained guide to this branch of mathematics, assuming the knowledge of basic concepts from an introductory abstract algebra course. And although various references are included throughout this work to established, advanced-level resources such as Weibel's ``An Introduction to Homological Algebra" \cite{weibel}, our work is written specifically with our above-mentioned audiences in mind, serving as a navigation tool, with clear and simple language, as well as a multitude of examples.

Indeed, using examples sourced from related fields such as group theory, topology and category theory, this work begins by presenting the notion of chain complexes (sequence of maps between Abelian groups or modules) in Section 2, to then move into introducing the concept of homology modules (an algebraic tool to compute the number of holes of a topological shape) in Section 3 and chain homotopies (the ability to continuously deform one shape to get another one) in Section 4. To illustrate the beauty of homological algebra in and of itself, this work also explores projective resolutions (a type of sequence of algebraic objects used to describe invariant properties) in Section 5 and ventures into the world of category theory by looking at Tor functors (a key tool that simplifies the analysis and classification of topological shapes) in Section 6.

However, before setting-off in the exploration of these various concepts, in the remaining portion of the introduction, we will review one of the most fundamental theorems from group theory, and use it alongside an example that will pave the way for the latter sections of this work. 

Effectively, algebra, on its most fundamental level, is mostly concerned with studying the structures of objects, and one of the most useful theorems to achieve this task is the first isomorphism theorem, which we recall. 
\begin{theorem}
Let $\varphi:X\rightarrow Y$ be a homomorphism (of groups G, rings R, or R-modules). Then, $X/\ker(\varphi)\cong \varphi(X)$.
\end{theorem}
\begin{proof}
\[\begin{tikzcd}
	X & Y \\
	{X/\ker\varphi}
	\arrow["\varphi", from=1-1, to=1-2]
	\arrow["\gamma"', from=1-1, to=2-1]
	\arrow["f"', from=2-1, to=1-2]
\end{tikzcd}\]
We follow the commutative diagram above to give a version of the proof, where $\gamma$ is the map sending $x\mapsto x+\ker\varphi$. Suppose that $\varphi:X\rightarrow Y$ is a homomorphism of the stated objects. Define the map $f:X/\ker\varphi\rightarrow \varphi(X)$ by $x+\ker\varphi\mapsto \varphi(x)$. First, we show that $f$ is well defined. If $x_1+\ker\varphi=x_2+\ker\varphi$, then $x_1-x_2\in \ker\varphi$, implying that $\varphi(x_1)=\varphi(x_2)$, so thus $f$ agrees on both arguments. One can show that $f$ is a homomorphism in any of the three categories posited by using the fact that $\varphi$ is a homomorphism. We see that $f$ is surjective; if $\varphi(x)\in \varphi(X)$, then $f(x+\ker\varphi)=\varphi(x)$. Finally, $f$ is injective. If $f(x_1+\ker\varphi)=f(x_2+\ker\varphi)$, then $\varphi(x_1)=\varphi(x_2)$. This implies that $\varphi(x_1-x_2)=0$, so hence $x_1-x_2\in \ker\varphi$. Evidently, $x_1+\ker\varphi=x_2+\ker\varphi$. 
\end{proof}
Now that we are equipped with this powerful theorem, let's examine the following motivating example.
\begin{example}
View $\mathbb{Z}$ as an Abelian group; let $\mathbb{Z}_2$ be the cyclic group of two elements, $\{e,x\}$. We define $\varphi:\mathbb{Z}\rightarrow \mathbb{Z}_2$ by \[
n\mapsto \begin{cases}
e,&n \equiv 0 \pmod{2} \\
x,& n \equiv 1 \pmod{2}
\end{cases}
\]where we leave the verification that $\varphi$ is a surjective group homomorphism to the reader. Note that $\ker(\varphi)=2\mathbb{Z}$; the first isomorphism theorem precisely asserts that $\mathbb{Z}/2\mathbb{Z}\cong \mathbb{Z}_2$.
\end{example}
An immediate way of interpreting this result is that the cosets $2\mathbb{Z}$ and $2\mathbb{Z}+1$, when endowed with the proper addition, form an Abelian group isomorphic to $\mathbb{Z}_2$. Indeed, this gives us insight into the structure of $\mathbb{Z}$, in particular, that $\mathbb{Z}$ can be divided into cosets, and that these cosets themselves form another Abelian group. Observe that $\mathbb{Z}\cong 2\mathbb{Z}$ as Abelian groups via the map $n\mapsto 2n$. Moreover, it is true in general that any two cosets in a quotient group are in bijection with each other; this is not hard to prove. 

This leads us to yet another interpretation. Effectively, we have used a subgroup $2\mathbb{Z}$ of $\mathbb{Z}$ to create a distinct Abelian group out of elements of $\mathbb{Z}$, where the elements of this new group are in bijection with $\mathbb{Z}$. However, what if we want to reverse this process? That is to say, suppose we are given $\mathbb{Z}_2$; can we find an Abelian group $B$ containing a subgroup that looks like $\mathbb{Z}$, such that the cosets when quotienting by this subgroup in $B$ form an Abelian group isomorphic to $\mathbb{Z}_2$? We say we are extending $\mathbb{Z}_2$ by $\mathbb{Z}$. In other words, we are trying to replace the elements of $\mathbb{Z}_2$ with sets that look like $\mathbb{Z}$.
\begin{example}
We rewrite our original example. Here, $\mathbb{Z}$ bijects onto $2\mathbb{Z}\subseteq \mathbb{Z}$ via the map $\psi(n)=2n$, and $\mathbb{Z}$ surjects onto $\mathbb{Z}_2$ via $\varphi$ as defined before:
\[
\mathbb{Z}\overset{\psi}{\rightarrow}\mathbb{Z}\overset{\varphi}{\rightarrow}\mathbb{Z}_2 \text{ .}
\]
Notice that $\im(\psi)=\ker(\varphi)=2\mathbb{Z}$. This is precisely what we want, for this means that $\mathbb{Z}/\im(\psi)\cong \varphi(\mathbb{Z})$. Furthermore, since $\psi:\mathbb{Z}\rightarrow \mathbb{Z}\ (\varphi:\mathbb{Z}\rightarrow \mathbb{Z}_2)$ is injective (surjective), $\mathbb{Z}\cong \im(\psi)$ and $\varphi(\mathbb{Z})=\mathbb{Z}_2$. We say that $\mathbb{Z}$ is an extension of $\mathbb{Z}_2$ by $\mathbb{Z}$.
\end{example}

The reason why these examples are interesting to analyze is because they illustrate the need to simultaneously use the injectivity of $\psi$, the surjectivity of $\varphi$, and the fact that $\im(\psi)=\ker(\varphi)$ to find the extension of our Abelian group. In fact, generally speaking, the problem of finding an extension of $\mathbb{Z}_2$ by $\mathbb{Z}$, denoted $B$, is precisely the task of determining when the sequence \[
\mathbb{Z}\overset{\psi}{\rightarrow}B\overset{\varphi}{\rightarrow}\mathbb{Z}_2
\]satisfies the aforementioned conditions. When it does, it is called a \textbf{short exact sequence}, and this is one of the key foundational concepts behind homological algebra. We will generalize this idea in the following section.

\section{Chain Complexes}
\label{sec:will}

The story of chain complexes, much like the story of most branches of mathematics, can be understood as successive enhancements on a concept to tackle ever greater challenges. In this case, chain complexes allow us to study a wide variety of structures that behave similarly to these short exact sequences.

For the first part of our journey, it will be helpful to introduce the following mental model: given a map \(\phi: X \to Y\) between two algebraic structures $X$ and $Y$ (such as groups or modules), we can view it in two separate ways. In one respect, the most straightforwards interpretation is that the map \(\phi\) \textit{creates} an \textit{object} (or element) of $Y$ using an \textit{object} (or element) of $X$. In this sense, we have that
\(\im(\phi) \subseteq Y\) represents the \emph{objects that are constructible by \(\phi\)}. In another, less obvious, respect, the map \(\phi\) associates each element $x \in X$ with an element $y \in Y$ according to some rule (dictated by the map), so one can also interpret \(\phi\) as \textit{testing} whether or not $x \in X$ satisfies a certain predicate that would cause it to map to the zero element $0 \in Y$. Under this secondary view, one can consider \(\ker(\phi)\) as representing the \emph{objects of $X$ that pass the test of \(\phi\)}, the test being whether these elements map to $0 \in Y$.

To visualize this duality in the essence of a function concretely, consider the two maps $\psi$ and $\varphi$ from the introduction section. Here, the map \(\psi: n \mapsto 2n\) \textit{constructs} the even integers from the integers, since every even integer can be expressed as $2n$ for some $n \in \mathbb{Z}$. On the other hand, the map \(\varphi: n \mapsto [n]_2\) \textit{tests} (or can be understood as a test to see) if an integer is divisible by 2, because an integer is divisible by 2 if and only if $[n]_2 = [0]_2$ (where $[x]_2$ is the equivalence class of $x$ modulo 2).

Philosophically, however, constructions are not always perfect. Consider, for example, the following unsuccessful attempt to construct the multiples of three:
\[\mathbb{Z} \overset{\psi}{\to} \mathbb{Z} \overset{\varphi}{\to} \mathbb{Z}_3\]
where \(\psi: n \mapsto 6n\) and \(\varphi: n \mapsto [n]_3\). Given a number \(n \in \im(\psi)\), \(n\) is in the kernel of \(\varphi\), because \(n\) can be written as \(6z, z \in \mathbb{Z}\), and \([6z]_3 = [6]_3 \cdot [z]_3 = [0]_3 \cdot [z]_3 = [0]_3\). This construction is not, however, perfect, because we are missing the odd multiples of 3, like 3, 9, etc., so \(\im(\psi) \subsetneq \ker(\varphi)\). We will see a formal definition soon, but when this happens we say that the sequence \(\mathbb{Z} \overset{\psi}{\to} \mathbb{Z} \overset{\varphi}{\to} \mathbb{Z}_3\) is \emph{not exact}. In this work, we will only be interested in examples where \(\im(\psi) \subseteq \ker(\varphi)\), which allows for the possibility of ``imperfect'' constructions $\psi$, but all outputs of $\psi$ still \textit{must pass the test} of $\varphi$.

Then, it might make sense to ask what happens if we chain multiple of these construction-test sequences together. For example, consider the sequence \(0 \overset{f}{\to} \mathbb{Z} \overset{\psi}{\to} \mathbb{Z} \overset{\varphi}{\to} \mathbb{Z}_2 \overset{g}{\to} 0\), where \(f: 0 \to \mathbb{Z}\) is the inclusion map from the trivial subgroup of \(\mathbb{Z}\), \(g: \mathbb{Z}_2 \to 0\) is the trivial homomorphism, and \(\psi\) and \(\varphi\) are defined as in the first example above. The first construction is \(f\) and the first test is \(\psi\), then \(\psi\) is also the second construction being tested by \(\varphi\), and then finally \(\varphi\) is a construction being tested by \(g\). One can check (and we will verify below) that \(\im(f) \subseteq \ker(\psi)\), \(\im(\psi) \subseteq \ker(\varphi)\), and \(\im(\varphi) \subseteq \ker(g)\), so this sequence is indeed of the form that we are interested in.

Now that we have seen sufficient motivation, we can begin to talk about the formal definition of what it means to have a chain complex:

\begin{definition}[Chain Complexes of Abelian Groups]
A \textbf{chain complex} $(\ch{A}, d)$ of Abelian groups is a collection of groups \(\{A_i\}_{i \in \mathbb{Z}}\) together with group homomorphisms (often called \textbf{boundary morphisms}) \(d_i: A_i \to A_{i - 1}\) for each \(i \in \mathbb{Z}\) such that \(d_{i - 1} \circ d_i = 0, \forall i \in \mathbb{Z}\).
\end{definition}
\begin{remark}
When unambiguous, mathematicians usually refer to a chain complex $(\ch{A}, d)$ as $\ch{A}$ or $A$ and the boundary morphisms are inferred.
\end{remark}

Notice that this formal definition does not \emph{quite} match the above examples, although we can recover what we were talking about before with the following general fact from Group Theory:

\begin{proposition}
\(d_{i - 1} \circ d_i = 0 \iff \im(d_i) \subseteq \ker(d_{i - 1})\)
\end{proposition}
\begin{proof}
Let \(y \in \im(d_{i})\). By definition, this means that \(\exists x \in A_i\) such that \(y = d_i(x)\). As a result, assuming the left-hand side is true, we have that \(d_{i - 1}(y) = (d_{i - 1} \circ d_i)(x) = 0(x) = 0\), so \(y \in \ker(d_{i - 1})\). The same logic is applicable for the backwards direction.
\end{proof}

Now, before moving on to ever higher levels of abstraction, it is worth taking a moment to pause and check whether or not the construction of a chain complex is indeed a natural extension of the concept of a short exact sequence of Abelian groups. To do this, we will need to introduce one more definition.

\begin{definition}
We say that a chain complex $(\ch{A}, d)$ is \textbf{exact} if \(\im(d_i) = \ker(d_{i - 1}), \forall i \in \mathbb{Z}\)
\end{definition}

\begin{proposition}
Given a sequence of Abelian group homomorphisms \(A \overset{f}{\to} B \overset{g}{\to} C\), this sequence is a short exact sequence \(0 \to A \overset{f}{\to} B \overset{g}{\to} C \to 0\) if and only if the following chain complex is exact:

\[\ldots{} \to 0 \to 0 \overset{\iota}{\to} A \overset{f}{\to} B \overset{g}{\to} C \overset{\tau}{\to} 0 \to 0 \to \ldots{} \]

where \(0 \to 0\) are the identities (denoted $\id_0$ below), \(\iota: 0 \to A\) is the inclusion from the trivial subgroup, and \(\tau: C \to 0\) is the trivial group homomorphism.
\end{proposition}

\begin{proof}
First, suppose that \(A \overset{f}{\to} B \overset{g}{\to} C\) is short exact. By definition, this means that \(\im(f) = \ker(g)\), \(f\) is injective, and \(g\) is surjective. Now, to check that the above chain complex is exact, we need to verify that the other group homomorphisms also satisfy the exactness criterion. Because \(\im(\id_0) = 0 = \ker(\id_0)\), so we just need to check the following four:
\begin{itemize}
\item \(\im(\id_0) \overset{?}{=} \ker(\iota)\)

\(\iota\) maps \(0\) (the only element) to \(0\), so \(\ker(\iota) = \{0\} = \im(\id_0)\).

\item \(\im(\iota) \overset{?}{=} \ker(f)\)

\(\iota\) maps \(0\) (the only element) to \(0\), so \(\im(\iota) = \{0\}\). \(f\) is injective, so \(\ker(f) = \{0\} = \im(\iota)\).

\item \(\im(g) \overset{?}{=} \ker(\tau)\)

\(\tau\) is the trivial homomorphsm so \(\ker(\tau) = C\). Because \(g\) is surjective, \(\im(g) = C = \ker(\tau)\).

\item \(\im(\tau) \overset{?}{=} \ker(\id_0)\)

\(\tau\) maps every element to \(0\), so \(\im(\tau) = \{0\} = \ker(\id_0)\).
\end{itemize}

Now, suppose that the given chain complex is exact. Because \(\im(\iota) = \ker(f)\) and \(\im(\iota) = \{0\}\), \(\ker(f) = \{0\}\), so \(f\) is injective. Because \(\im(g) = \ker(\tau)\) and \(\ker(\tau) = C\), \(\im(g) = C\), so \(g\) is surjective. Thus, \(A \overset{f}{\to} B \overset{g}{\to} C\) is short exact.
\end{proof}

It turns out, however, that we can make sense of chain complexes even outside the context of Abelian groups. Recall that the modules over the ring $\mathbb{Z}$ are precisely the Abelian groups, with the $\mathbb{Z}$-action defined as the group exponentiation (which in the context of Abelian groups is often written in terms of $\mathbb{Z}$-multiplication). Therefore, a reasonable conjecture would be that chain complexes also make sense over $R$-modules for any ring \(R\). This conjecture is perhaps informed by the fact that other important results about Abelian groups, like the classification of finitely-generated Abelian groups, also generalize to $R$-modules. And, in fact, that this conjecture is correct, and we can upgrade our definition from above as follows:

\begin{definition}[Chain Complexes of $R$-modules]
A \textbf{chain complex} $(\ch{C}, d)$ of modules over a ring \(R\) is a collection of $R$-modules \(\{C_i\}_{i \in \mathbb{Z}}\) together with $R$-module homomorphisms (often called \textbf{boundary morphisms}) \(d_i: C_i \to C_{i - 1}\) for each \(i \in \mathbb{Z}\) such that \(d_{i - 1} \circ d_i = 0, \forall i \in \mathbb{Z}\).
\end{definition}

\begin{remark}
Chain complexes over $R$-modules are \textit{incredibly} rich and interesting to study. In fact, the other sections of this paper will focus almost entirely on chain complexes over $R$ modules, so outside of this section whenever you see the term ``chain complex'', you should assume that we are discussing $R$-modules.
\end{remark}

When studying chain complexes, we are often interested in how they behave relative to other chain complexes, so we need a way to translate between chain complexes. We can do this by way of a chain complex map:

\begin{definition}
A \textit{chain complex map} $u: (\ch{C}, d) \to (\ch{D}, d')$ is a collection of $R$-module homomorphisms $\{u_i: C_i \to D_i\}_{i \in \mathbb{Z}}$ such that the following diagram commutes for each $n \in \mathbb{Z}$:
\[
\begin{tikzcd}
\ldots{} \arrow[r, "d_{n + 2}"]
& C_{n + 1} \arrow[r, "d_{n + 1}"] \arrow[d, "u_{n + 1}"]
& C_n \arrow[r, "d_n"] \arrow[d, "u_n"]
& C_{n - 1} \arrow[r, "d_{n - 1}"] \arrow[d, "u_{n - 1}"]
& \ldots{} \\
\ldots{} \arrow[r, "d'_{n + 2}"]
& D_{n + 1} \arrow[r, "d'_{n + 1}"]
& D_n \arrow[r, "d'_n"]
& D_{n - 1} \arrow[r, "d'_{n - 1}"]
& \ldots{} \\
\end{tikzcd}
\]
\end{definition}

Now that we have a notion of a map between chain complexes, one would hope that these maps are well-behaved. Indeed, one can check that the identity morphism between chain complexes is the morphism which is the identity morphism at each level, there is an associative composition formed by the element-wise compositions of each $R$-module homomorphism $u_i: C_i \to D_i$, and composition respects identities in the proper way.

At this point, we have defined most of the tools necessary for the study which will be undertaken in the rest of the paper. For the remaining portion of this section, we will explore more exotic chain complexes and conclude with a beautiful result that ties all of these different examples together into one all-encompassing definition.

Effectively, using our chain complex maps, it is possible to define a chain complex over chain complexes, analogously as above to be an ordered pair $(C_{\bullet, \bullet}, d)$, where $C_{n, \bullet} \in \Ch, \forall n \in \mathbb{Z}$ and $d_{n - 1} \circ d_n = 0, \forall n \in \mathbb{Z}$, and this chain complex is exact if $\im(d_n) = \ker(d_{n - 1}), \forall n \in \mathbb{Z}$. Here, $\Ch$ is used to denoted the category of chain complexes, making $C_{n, \bullet}$ one of its objects.

In this context, $0$ refers to the chain complex map $\{0: C_{n, i} \to 0\}_{i \in \mathbb{Z}}$, $\ker(d_n)$ refers to the chain complex formed by taking $\ker(d_{n,i})$ at each position $i$ and restricting and co-restricting $d_{n,i}: \ker(d_{n,i}) \to \ker(d_{n, i - 1})$ for each $i$. This restriction and co-restriction is well-defined because $\im(d_{n,i}) \subseteq \ker(d_{n, i - 1})$, so each element in the domain maps to an element in the codomain. Similarly, $\im(d_n)$ refers to the chain complex formed by taking $\im(d_{n, i + 1})$ at each position $i$ and restricting and co-restricting $d_{n, i}: \im(d_{n, i + 1}) \to \im(d_{n, i})$, which is well-defined because $\im(d_{n, i + 1}) \subseteq C_{n, i}$ and $d_{n, i}$ only maps to values in $\im(d_{n, i})$.

This definition seems unmotivated right now, and in fact we do not formally define what the kernel or the image of a chain complex is, but this is enough to prove the following result. One can check that our notion of kernels and images coincides with the more general notion defined later by satisfying the universal property.

\begin{proposition}
A sequence $0 \to \ch{A} \overset{f}{\to} \ch{B} \overset{g}{\to} \ch{C} \to 0$ of chain complexes is exact in $\Ch(\rmod{R})$ if and only if each sequence $0 \to A_n \overset{f_n}{\to} B_n \overset{g_n}{\to} C_n \to 0$ is exact in $\rmod{R}$. Here, $\Ch(\rmod{R})$ is used to denote the category of chain complexes over $R$-modules.
\end{proposition}
\begin{proof}
Using the above definition of the kernel and image, the sequence of chain complexes is exact if and only if $\ker(f) = 0$, $\im(f) = \ker(g)$, and $\ker(g) = 0$, which happens if and only if for all $i \in \mathbb{Z}$ (i.e. at all positions in the respective chain complexes), $\ker(f_i) = 0$, $\im(f_i) = \ker(g_i)$, and $\ker(g_i) = 0$, which happens if and only if all of the individual sequences are exact.
\end{proof}

At first, it may seem surprising that we can form chain complexes of chain complexes, and that these are well-behaved in a similar way as chain complexes over Abelian groups or $R$-modules. It turns out that many other mathematical objects support chain complexes as well. And furthermore, these different ``chain complexes'' are not merely disparate structures with a common name, but they really fundamentally behave in very similar ways.

Because of these similarities, mathematicians began searching for the underlying structure that unites all of these ideas together. Like many results unifying the study of similar structures in different contexts, the true universal definition of a chain complex comes to us via Category Theory. Over the following definitions, we will build up to a characterization of the kinds of objects that support the notion of a chain complex. To start with, we need a rigorous way to talk about different ``kinds'' of objects and morphisms, like Abelian groups and their homomorphisms, chain complexes over $R$-modules and their chain complex maps. As was hinted before, Category Theory provides us with such a definition. To start with:

\begin{definition}
A \textbf{category} $\mathcal{C}$ consists of the following:
\begin{itemize}
    \item An unordered collection $\Ob(\mathcal{C})$ of \textbf{objects}.
    \item An unordered collection $\hom{\mathcal{C}}(A, B)$ of \textbf{morphisms} for each pair of objects $A, B$. An element $f \in \hom{\mathcal{C}}(A, B)$ is often denoted $f: A \to B$.
    \item A \textbf{composition operator} $\circ$ such that for any $f: A \to B$ and $g: B \to C$, $g \circ f: A \to C$.
\end{itemize}
such that the following properties hold:
\begin{itemize}
    \item $\circ$ is associative
    \item For each object $X \in \Ob(\mathcal{C})$ there exists a morphism $\id_X$ such that for all $f: X \to Y$, $f \circ \id_X = f$, and for all $f: W \to X$, $\id_X \circ f = f$.
\end{itemize}
\end{definition}

As a result, using the above, the proper definition of a category of chain complexes of $R$-modules can be concretely written as follows:

\begin{definition}
The \textbf{category of chain complexes of $R$-modules} $\Ch = \Ch(\rmod{R})$ is the category whose objects are chain complexes $(\ch{C}, d)$, whose morphisms are chain complex maps, and whose compositions are defined as above.
\end{definition}

One cam similarly define the categories $\Ab$ of Abelian groups and $\rmod{R}$ for modules over a ring $R$.

\begin{definition}[Additive Categories]
We say that a category $\mathcal{A}$ is \textbf{additive} if the following conditions hold:
\begin{enumerate}[(i)]
    \item Each hom-set $\hom{\mathcal{A}}(A, B)$ is an Abelian group, where the binary operation $+$ respects composition: i.e., for all morphisms $e \in \hom{\mathcal{A}}(A, B)$, $f, g \in \hom{\mathcal{A}}(B, C)$, and $h \in \hom{\mathcal{A}}(C, D)$, we have that $h \circ (f + g) \circ e = h \circ f \circ e + h \circ g \circ e$.
    
    \item $\mathcal{A}$ has an object $0$ such that for all other objects $X$, there exists unique morphisms $0 \to X$ and $X \to 0$, which we call the \textbf{zero element}.
    
    \item For every pair of objects $A$ and $B$ in $\mathcal{A}$, the product $A \times B$ is an object in $\mathcal{A}$.
\end{enumerate}
\end{definition}

An additive category gives us enough structure to define what kernels, cokernels, and images are:
\begin{definition}
In an additive category $\mathcal{A}$, given a morphism $f: B \to C$, we define the \textbf{kernel of $f$} (if it exists) to be a map $i: A \to B$ such that $f \circ i = 0$, and given any other map $i': A' \to B$ such that $f \circ i' = 0$, there exists a unique map $\phi: A' \to A$ such that $i' = i \circ \phi$. \\

\noindent Dually, given a morphism $f: B \to C$, we define the \textbf{cokernel of $f$} (if it exists) to be a map $e: C \to D$ such that $e \circ f = 0$, and given any other map $e': C \to D'$ such that $e' \circ f = 0$, there exists a unique map $\phi: D \to D'$ such that $e' = \phi \circ e$. \\

\noindent Finally, we define the \textbf{image of $f$} to be $\ker(\coker(f))$.
\end{definition}
The above definition (which is already an elaboration of the definition given by \cite{weibel}) is still rather incomplete on its own, and deserves to be unpacked further. First of all, when we talk about a zero morphism $0: A \to B$, we really mean the unique morphism $A \to 0$ composed with the unique morphism $0 \to B$. In the categories we are used to working with, this corresponds to the $0$ homomorphism we are familiar with.

A careful observer, however, might argue that this definition still does not match the definition usually given for a kernel or cokernel in $\rmod{R}$. In $\rmod{R}$, there is a clear notion of a ``subobject'', so we can understand $\ker(f)$ as a subobject of the domain. If we let $\mathcal{K}$ indicate this object, in the above definition $\ker(f)$ is really the $R$-module homomorphism $\mathcal{K} \injectsto A$. For arbitrary additive categories, it makes sense to define the kernel to be this morphism rather than the object because we don't always have a convenient notion of a ``subobject'', so $A$ is not necessarily a ``subobject'' of $B$ in the sense we would want it to be.

This leaves the definition of the image to be demystified. The following proposition clarifies somewhat why this definition of the image coincides with our usual idea of the image of a map in $\rmod{R}$.

\begin{proposition}
In $\rmod{R}$, given a morphism $f: A \to B$, let the image of $f$ be denoted by $\mathcal{I}$. Then $\im(f) = \ker(\coker(f)): \mathcal{I} \to B$.
\end{proposition}
\begin{proof}
Consider the canonical projection $\gamma: B \to B / \mathcal{I}$. Because $f(x) \in \mathcal{I}, \forall x \in A$, $(\gamma \circ f)(x) = \mathcal{I}, \forall x \in A$, so it is the 0 morphism. I claim that $\gamma$ is the cokernel of $f$. Let $g: B \to C$ such that $g \circ f = 0$. Because $g \circ f = 0$, we must have $g(i) = 0, \forall i \in \mathcal{I}$, because by definition for each $i \in \mathcal{I}$ there exists an $x \in A$ such that $f(x) = i$. Thus, $\mathcal{I} \subseteq \ker(g)$, so by the universal property of the quotient $\exists! \bar{g}$ such that $g = \bar{g} \circ \gamma$, so thus $\gamma$ satisfies the universal property of the cokernel.

Now, consider $\ker(\gamma)$. I claim that $\iota: \mathcal{I} \injectsto B$ is the kernel of $\gamma$. Let $i \in \mathcal{I}$. $(\gamma \circ \iota)(i) = \gamma(i) = \mathcal{I}$, so $\gamma \circ \iota = 0$. Now let $h: A \to B$ such that $\gamma \circ h = 0$. Because $\gamma \circ h = 0$, we must have $(\gamma \circ h)(x) = 0$ for all $x \in A$, which means that $\gamma(h(x)) = \gamma(x) = \mathcal{I}$, so $x \in \mathcal{I}, \forall x \in A$. This means that we can co-restrict $h$ to get a new $R$-module homomorphism $\hat{h}: A \to \mathcal{I}$ such that $h = \iota \circ \hat{h}$, and $\hat{h}$ is unique because $h$ and $\hat{h}$ map all of the same values, just with a different codomain. Thus, $\iota$ satisfies the universal property of the kernel.

Therefore, $\im(f) = \ker(\coker(f)): \mathcal{I} \to B$.
\end{proof}

Thus, the image $\mathcal{I}$ that we are used to dealing with is really just the object associated with the image as defined in a categorical sense, so the two definitions go hand-in-hand.

An additive category gives us all of the structure that we need in order to obtain chain complexes and check whether or not they are exact:
\begin{definition}[Chain Complexes in an Additive Category]
A \textbf{chain complex} $(\ch{A}, d)$ in some additive category $\mathcal{A}$ is a collection of objects \(\{A_i\}_{i \in \mathbb{Z}}\) where $A_i \in \mathcal{A}, \forall i \in \mathbb{Z}$, together with morphisms (often called \textbf{boundary morphisms}) \(d_i: A_i \to A_{i - 1}\) for each \(i \in \mathbb{Z}\) such that \(d_{i - 1} \circ d_i = 0, \forall i \in \mathbb{Z}\).
\end{definition}
\begin{definition}
A chain complex $(\ch{A}, d)$ in some additive category $\mathcal{A}$ is \textbf{exact} if $\im(d_i) = \ker(d_{i - 1}), \forall i \in \mathbb{Z}$.
\end{definition}

In practice, however, additive categories do not provide quite enough structure to be interesting to study, because we are not guaranteed that images and kernels exist, so we have no way of measuring whether or not a chain complex is exact. To remedy this, mathematicians restrict their focus to chain complexes over Abelian categories, where Abelian categories are categories which satisfy a few additional constraints\footnote{A full definition of an Abelian category can be obtained in \cite{weibel}, which includes some extra details that are not worth mentioning here.}, one of which is that for each morphism $f$, $\ker(f)$ and $\coker(f)$ exist in $\mathcal{A}$. This also guarantees that $\im(f)$ exists, because $\im(f)$ is constructed out of kernels and cokernels. Some examples of Abelian categories include $\Ab$ and $\rmod{R}$.

Finally, to return to the question of why $\Ch$ behaves so well, it all relies on the following theorem:

\begin{theorem}
$\Ch(\rmod{R})$ is an Abelian category. Furthermore, for any Abelian category $\mathcal{A}$, the chain complexes over $\mathcal{A}$ form a category $\Ch(\mathcal{A})$, and $\Ch(\mathcal{A})$ is Abelian.
\end{theorem}
\begin{proof}
We can define the morphisms between chain complexes in the same way as we defined them above for $\Ch(\rmod{R})$, which one can check defines a category with well-behaved identity morphisms and composition.

To see that $\Ch(\mathcal{A})$ is an additive category, observe that we can lift the binary operation $+$ for morphisms of $\mathcal{A}$ to $\Ch(\mathcal{A})$ by adding chain complex maps levelwise. This means that composition distributes over addition because it does so levelwise. $\Ch(\mathcal{A})$ also has a 0 object, which is given by the chain complex $\ldots{} \to 0 \to 0 \to \ldots{}$. The unique morphisms to and from any chain complex $\ch{A}$ are given by the chain maps $\{0: A_i \to 0\}_{i \in \mathbb{Z}}$ and $\{0: 0 \to A_i\}_{i \in \mathbb{Z}}$.

The proof that products of pairs exist in $\Ch(\mathcal{A})$ and the proof that $\Ch(\mathcal{A})$ is Abelian are less enlightening and require some more machinery from category theory than has been defined thus far, so the curious reader can refer to Weibel \autocite{weibel}.
\end{proof}

Equipped with a rigorous set of definitions, in the section that follows below, we will explore some interesting applications of chain complexes, and how these help encode information about topological shapes.

\section{Homology}
\label{sec:andy}
The origins of homology lie in the study of geometries, and realizing, in particular, that their topologies can be differentiated from one another and be classified by examining their holes. This is, for instance, what makes a Boston creme doughnut different from a regular glazed doughnut, at least in topological terms. But to understand a shape's holes, one must first study the holes' $boundaries$, in all possible dimensions, including in 0D (points), 1D (edges), and 2D (faces). For example, when looking at a solid 4-face pyramid in 3D, we can define its boundary as the hollow pyramid of four 2D faces. This definition might not seem intuitive at first, although a useful analogy is to think of wrapping the geometry with paper, and considering the wrap as the boundary. Similarly, when looking at each of these four 2D faces that make the boundary of the 3D pyramid, we find that each of these are in turn bounded by three 1D edges (again, enclosing the faces); and each one of these 1D edges are bounded by two 0D points each! This example\footnote{Without delving too deep into concepts of topology, this example helps demonstrate what simplices are, generalizations of 2D triangles to multiple dimensions. In essence, a 0-simplex is a point, a 1-simplex is a line, a 2-simplex is a triangle, and so on. The reason why studying triangles is so important in the field of topology is due to triangulation; Hungarian mathematician Tibor Rado demonstrated that essentially most 2D surfaces can be constructed, or triangulated, using 2-simpleces. This result, for which an analogous version exists in 3D, is extremely useful in a variety of disciplines, such as in meshing applications for finite element analysis. More on this topic can be seen in \cite{tom}, alongside additional fascinating facts such as Euler's formula.} illustrates two crucial points about the study of the topologies of geometries. The first is that the dimension of the boundary is lower than that of the studied shape itself. The second, and more important point, is that the study of boundaries provides enough information to create building instructions for any shape. 

Incredibly, it is precisely chain complexes that allow us to take this information about a shape's boundaries in all possible dimensions and construct its building instructions. This would permit obtaining a sequence of maps such as:
\[
Points\overset{\psi_1}{\rightarrow}Edges\overset{\psi_2}{\rightarrow}Faces\overset{\psi_3}{\rightarrow}Volume
\]

In the pyramid example from above, we had a sequence relating four points, six edges, four faces and one volume. Notice that this sequence was limited to dimensions we were familiar with (0D, 1D, 2D, 3D), but we can continue up to an arbitrary number of dimensions, where visualizing boundaries is certainly not trivial. 

Now, the question is, how do we relate boundaries and chain complexes to holes and homologies? Under this line of thought, following the example provided by Kelsey Houston \autocite{quanta}, holes are defined by ``not possessing boundaries", or in other words, arise whenever we have a fully closed hollow boundary/loop. As an example, if we had 3 0D points but just 2 1D line segments connecting these points, then adding a third 1D line segment would close the loop, creating a hollow triangle. This gives rise to a hole, given that the boundary of each segment (the 0D points) disappear. As a result, chain complexes, by storing information about the boundaries of shapes in each dimension, provide a framework from which holes can be determined, by examining closely each boundary.

Homologies, or more precisely homology modules, provide precisely the means through which we can use chain complexes to calculate the number of holes at each dimension. And not surprisingly, these define the core of homological algebra, the topic of this paper, which as we know, is the branch of mathematics that acts as the translator between the world of shapes and topologies and the more computable world of algebra.

To construct homology modules, consider the following definition:
\begin{definition}
Let $(\ch{C}, d)$ be a chain complex, with maps $d_n:C_n\rightarrow C_{n-1}$, as shown below:
\[...{\rightarrow}C_{n+1}\overset{d_{n+1}}{\rightarrow}C_{n}\overset{d_n}{\rightarrow}C_{n-1}\overset{d_{n-1}}{\rightarrow}C_{n-2}\overset{d_{n-2}}{\rightarrow}...\text{ .}\]

Then, define $Z_n(C) = \ker(d_n)$ as the module of \textbf{n-cycles} of $C$ and $B_n(C) = \im(d_{n+1})$ as the module of \textbf{n-boundaries} of $C$. The $\mathbf{n^{th}}$ \textbf{homology module} of the chain complex $C$ is then defined by the quotient of the $n^{th}$-cycle and $n^{th}$-boundary of $C$, or more precisely $H_n(C) = Z_n(C)/B_n(C) = \ker(d_n)/\im(d_{n+1})$

Notice that $\forall n \in \mathbb{Z}$, we have that $0 \subseteq B_n(C) \subseteq Z_n(C) \subseteq C_n$. As a result, a morphism $u:C\rightarrow D$ between two chain complexes $C$ and $D$ would need to preserve these structures, and map, $\forall n \in \mathbb{Z}$, the $n^{th}$-boundaries of $C$ to $n^{th}$-boundaries of $D$, the $n^{th}$-cycles of $C$ to the $n^{th}$-cycles of $D$, and most importantly, would need to map the  $n^{th}$-homology groups of $C$ to those of $D$. One such morphism $u:C\rightarrow D$ is called a \textbf{quasi-isomorphism} if the maps between homology modules, ${\phi_n}:H_n(C)\rightarrow H_n(D)$, are all isomorphisms. 
\end{definition}
To illustrate this definition, consider the following proposition. 

\begin{proposition}
For every chain complex $(\ch{C}, d)$, we have that the following are equivalent:\begin{enumerate}[(i)]
    \item $C$ is exact, that is, at every $C_n$;
    \item $C$ is acyclic, that is, $H_n(C) = 0$, $\forall n \in \mathbb{Z}$;
    \item The map $u:0\rightarrow C$ is a quasi-isomorphism, where $0$ represents the chain complex composed of zero modules and zero maps.
\end{enumerate}
\end{proposition}

\begin{proof}
$(i) \implies (ii)$ By assumption of $(i)$, $C$ being exact means, by definition, that $\im(d_{n+1}) = \ker(d_n)$ at every $n$, so then we automatically get that 

\noindent $\ker(d_n)/\im(d_{n+1})= H_n(C) = 0$ as well, making $C$ acyclic.

$(ii) \implies (iii)$ By assumption of $(ii)$, we have that for the chain complex $C$, $H_n(C) = 0$ for all the positions $n$. Furthermore, given that $0$ is the chain complex composed of zero modules and zero maps, then it must also follow that $H_n(0) = 0$ as well (this is not surprising, since $\im(d_{n+1}) = \ker(d_n) = 0$). As a result, this means that the maps between the homology modules, ${\phi_n}:H_n(0)\rightarrow H_n(C)$ are all trivial maps, mapping $0\rightarrow 0$. Thus, these are all isomorphisms, making the map $u:0\rightarrow C$ a quasi-isomorphism.

$(iii) \implies (i)$ By definition of quasi-isomophism, we have that all maps ${\phi_n}:H_n(0)\rightarrow H_n(C)$ must be isomorphisms. Hence, because of the structure of the chain complex $0$ and by assumption of $(iii)$, then it is clear that the mapping $H_n(0) = 0\rightarrow H_n(C)$ must be trivial too in order for it to be an isomorphism (we can't map $0$ to more than one element). Hence, $H_n(C)$ has to necessarily be equal to $0$ for all $n$, proving that the chain complex $C$ is exact. 
\end{proof}

Now that the main definition for homology modules has been provided, below we provide a computational example, which should clarify and illustrate how exactly we can use them.

\begin{example}
Recall the sequence from Example 1.2, which had the maps $\psi:\mathbb{Z}\rightarrow \mathbb{Z}$, mapping $n\in \mathbb{Z} \mapsto 2n\in \mathbb{Z}$, and $\varphi:\mathbb{Z}\rightarrow \mathbb{Z}/2\mathbb{Z}$, mapping $z\in \mathbb{Z}$ to its equivalence class in $\mathbb{Z}/2\mathbb{Z}$. This resulted in: 
\[
\mathbb{Z}\overset{\psi}{\rightarrow}\mathbb{Z}\overset{\varphi}{\rightarrow}\mathbb{Z}/2\mathbb{Z}
\]
As a result, if we have that $\im(\psi) = \ker(\varphi)$, then this means that an integer $k$ is divisible by 2 $\iff$ $k = 2q$, for some other integer $q$. In an more fundamental and philosophical level, the integers $k = 2q$ represent the things that can be constructed (in this case, with the map $\psi$), while the integers that are divisible by 2 are what we can test (in this case, through the map $\varphi$). Thus, the homology module of this sequence, $\ker(\varphi)/\im(\psi)$, captures precisely how successful we were in measuring what we wanted to test. If, for instance, $\psi$ mapped $n\in \mathbb{Z} \mapsto 4n\in \mathbb{Z}$, then the homology module would certainly be non-zero, since we are failing to capture certain multiples of 2 (like 2, 6, 10, ...).

Overall, this provides another way of understanding homologies that is not strictly related to topological spaces.
\end{example}

One interesting feature to mention about homology modules is that we can extend their definition to products and direct sums of chain complexes. To see this with more clarity, first consider the following definition:

\begin{definition}
Let $(\ch{{C_i}}, d_i)$ be a family of chain complexes of $R$-modules, where $i \in I$. Then, the \textbf{product} $\prod_{i \in I} C_i$ and \textbf{direct sum} (coproduct)  $\bigoplus_{i \in I} C_i$ both exist in $\Ch$ (the category of chain complexes), and are defined degree-wise. This means we have: 
\begin{align*}
\prod_{i \in I} d_{i,n}:\prod_{i \in I} C_{i,n}\rightarrow \prod_{i \in I} C_{i,n-1} && \bigoplus_{i \in I} d_{i,n}:\bigoplus_{i \in I} C_{i,n}\rightarrow \bigoplus_{i \in I} C_{i,n-1}
\end{align*}
where $d_{i,n}$ represent the maps that go between $C_{i,n}$ and $C_{i,n-1}$, for a fixed $i$. 
\end{definition}

Hence, this allow us to introduce the following proposition, which asserts that we can compute the homologies of products and direct sums of chain complexes.

\begin{proposition}
Products and directs sums of chain complexes commute with homology modules. In other words, we have that:
\begin{align*}
\prod_{i \in I} H_n(C_{i}) \cong H_n(\prod_{i \in I} C_{i}) &&
\bigoplus_{i \in I} H_n(C_{i}) \cong H_n(\bigoplus_{i \in I} C_{i})
\end{align*}
As can be seen, the power of this proposition lies in the fact that taking the homology module of a product or direct sum of chain complexes is the same, up to isomorphism, than taking the homology module of each chain complex, and then the product or direct sum. 
\end{proposition}

\begin{proof}
The proof follows from simply applying the universal property of products and direct sums, and utilizing the definition of the homology module.
\end{proof}

To finalize the discussion and applications of homology module, we can introduce one last concept, which is that of long exact sequences of chain complexes. In essence, through the mechanism that will be described below, it is possible to take a (short exact) sequence of chain complexes, compute the homology module at each $n^{th}$ position of each chain, and obtain a sequence by concatenating all of these homology modules together. More formally said, consider the following theorem:

\begin{theorem}

Let $0\rightarrow A \overset{f}\rightarrow B \overset{g}\rightarrow C \rightarrow 0 $ be a short exact sequence of the chain complexes $(\ch{A}, d_A)$, $(\ch{B}, d_B)$ and $(\ch{C}, d_C)$. Then, we naturally obtain maps $\partial_n:H_n(C)\rightarrow H_{n-1}(A)$ (between the homology at the $n^{th}$ position of $C$ and the $(n-1)^{th}$ position of $A$), which are known as connecting homomorphisms. This results in the following long exact sequence of homology module:

\[\begin{tikzcd}
	\ldots & H_{n+1}(A) & H_{n+1}(B) \arrow[d, phantom, ""{coordinate, name=Y}] & H_{n+1}(C)&\\
	& H_n(A) & H_n(B) \arrow[d, phantom, ""{coordinate, name=Z}] & H_n(C)&\\
	& H_{n-1}(A) & H_{n-1}(B) & H_{n-1}(C) & \ldots
	\arrow[from=1-1, to=1-2]
	\arrow["f", from=1-2, to=1-3]
	\arrow["g", from=1-3, to=1-4]
	\arrow[from=1-4, to=2-2,
	rounded corners,
to path={ -- ([xshift=2ex]\tikztostart.east) \tikztonodes
|- (Y) [near end]
-| ([xshift=-2ex]\tikztotarget.west)
-- (\tikztotarget)},
"\partial_{n + 1}"
	]
	\arrow["f", from=2-2, to=2-3]
	\arrow["g", from=2-3, to=2-4]
	\arrow[from=2-4, to=3-2,
rounded corners,
to path={ -- ([xshift=2ex]\tikztostart.east)\tikztonodes
|- (Z) [near end]
-| ([xshift=-2ex]\tikztotarget.west)
-- (\tikztotarget)},
"\partial_n"
	]
	\arrow["f", from=3-2, to=3-3]
	\arrow["g", from=3-3, to=3-4]
	\arrow[from=3-4, to=3-5]
\end{tikzcd}\]

\end{theorem}

The key to proving this theorem lies in both demonstrating the existence and defining the structure of these connecting homomorphism maps $\partial_n$, since these essentially permit the realization of the sequence of homology module shown above. This task is accomplished by a famous lemma known as the ``Snake Lemma\footnote{The connecting homomorphism $\partial$ from $\ker(c)$ to $\coker(a)$, taking the shape of a snake, is what gives the lemma its nickname, as is seen in the diagram below.}", whose statement is introduced below. 

\begin{lemma}
[Snake Lemma] Consider the following commutative diagram of the R-modules $M_1$, $M_2$, $M_3$ and $M_1'$, $M_2'$, $M_3'$:
\[\begin{tikzcd}
{}
& M_1 \arrow[r, "f"] \arrow[d, "a"]
& M_2 \arrow[r, "g"] \arrow[d, "b"]
& M_3 \arrow[d, "c"] \arrow[r]
& 0
\\
0 \arrow[r]
& M_1' \arrow[r, "f'"] 
& M_2' \arrow[r, "g'"] 
& M_3' 
& {}
\end{tikzcd}\]

If we have that the two rows in the diagram are exact, then we obtain the following exact sequence:
\[
\ker(a) \rightarrow \ker(b)\rightarrow \ker(c) \overset{\partial}\rightarrow \coker(a) \rightarrow \coker(b)\rightarrow \coker(c)
\]

\end{lemma}

Here, $\partial$ is precisely the connecting homomorphism mentioned in Theorem 3.1, and is given by $\partial = f'^{-1} b g^{-1}$. In addition, recall that $\coker(a) = M_1'/\im(a)$, $\coker(b) = M_2'/\im(b)$, and $\coker(c) = M_3'/\im(c)$. Now, to better understand the origins of this connecting homomorphism map $\partial$ and why it is defined in this way, it is best to look at the following diagram:
 \[\begin{tikzcd}
{}
& \ker(a) \arrow[r] \arrow[d, hook]
& \ker(b) \arrow[r] \arrow[d, hook, ""{coordinate, name=Z}]
& \ker(c) \arrow[d, hook] \arrow[dddll,
"\partial",
rounded corners,
to path={ -- ([xshift=2ex]\tikztostart.east) \tikztonodes
|- (Z) 
-| ([xshift=-2ex]\tikztotarget.west)
-- (\tikztotarget)}] 
\\
& M_1 \arrow[r, "f"] \arrow[d, "a"]
& M_2 \arrow[r, "g"] \arrow[d, "b"]
& M_3 \arrow[d, "c"] \arrow[r]
& 0
\\
0 \arrow[r]
& M_1' \arrow[r, "f'"] \arrow[d]
& M_2' \arrow[r, "g'"] \arrow[d]
& M_3' \arrow[d]
& {}
\\
{}
& \coker(a) \arrow[r] 
& \coker(b) \arrow[r] 
& \coker(c)
& {}
\end{tikzcd}\]

\begin{proof}
The proof for the Snake Lemma uses the concept of diagram chasing. However, this elaboration is quite involved, so we refer to the complete proof written by Bob Gardner \autocite{gardner}. 
\end{proof}

Now, notice that when we replace $M_1$, $M_2$, $M_3$ for the homology modules $H_n(A)$, $H_n(B)$, $H_n(C)$, and then replace $M_1'$, $M_2'$, $M_3'$ for the homology modules $H_{n-1}(A)$, $H_{n-1}(B)$, $H_{n-1}(C)$, this lemma automatically gives us the desired connecting homomorphism map $\partial_n:H_n(C) \rightarrow H_{n-1}(A)$. Here, recall that $A$, $B$ and $C$ were chain complexes part of the short exact sequence given by the assumption of Theorem 3.1. Thus, one can check that the diagram in the lemma commutes and that the rows are exact, due to this assumption.

One interesting result that follows from the application of Theorem 3.1 can be found below:

\begin{proposition}
Let $0\rightarrow A \overset{f}\rightarrow B \overset{g}\rightarrow C \rightarrow 0 $ be a short exact sequence of chain complexes $(\ch{A}, d_A)$, $(\ch{B}, d_B)$ and $(\ch{C}, d_C)$. Then, if any two of these chain complexes is exact, so is the third. 
\end{proposition}

\begin{proof}
 First, notice that by Theorem 3.1, we can construct a sequence of homology modules as given above that forms a long exact sequence. However, recall that by Proposition 3.1 above, if $A$ is a short exact sequence, then $A$ is acyclic, so $H_n(A) = 0$, $\forall\ n \in \mathbb{Z}$. Thus, if two of the chain complexes making the short exact sequence were exact, then we would get that the long exact sequence of homology modules would be composed of $0s$ in two-thirds of all positions. Denoting this long exact sequence as $L$, then: 
\[L = \begin{cases}
H_n(C) \text{, if } n \equiv 0 \mod 3 \\
0 \text{, otherwise}
\end{cases}\]

Graphically, this is illustrated as follows:
 \[... \rightarrow 0 \rightarrow H_{n+1}(C) \rightarrow 0 \rightarrow 0 \rightarrow H_n(C) \rightarrow 0 \rightarrow 0 \rightarrow H_{n-1}(C) \rightarrow 0 \rightarrow ...\]
 Here, we are assuming without loss of generality that the chain complex $C$ is not (necessarily) exact, while $A$ and $B$ are. Thus, in order to show the exactness of $C$, then we would need $H_n(C) = 0$, $\forall n \in \mathbb{Z}$. To do so, we can make use of the fact that if we have a short exact sequence $0\rightarrow A \overset{f}\rightarrow B \rightarrow 0 $, then $f$ is an isomorphism. This is due to exactness at $A$ requiring the image of the zero map to be equal to the kernel of $f$, while exactness at $B$ enforcing the image of $f$ to be equal to the kernel of the zero map. As a result, applying this fact to our problem, by letting $B = 0$ and $A = H_n(C)$, we get that due to $f$ being an isomorphism, then $A$ must be equal to $0$. Thus, $H_n(C) = 0$, $\forall n \in \mathbb{Z}$, showing that the chain complex $C$ had to be exact to begin with. 
 
\end{proof}

Overall, this section allowed us to investigate and understand the importance that homological algebra has in the study of topologies, through chain complexes and homology modules. In essence, we can differentiate between topological spaces/shapes by studying their boundaries and their respective number of holes in all possible dimensions, whether that is in 0D, 1D, 2D or beyond. Now, in the section that follows, we will continue building this algebraic toolbox so that we can classify shapes and objects in topological spaces in an even more robust way.

\section{Homotopy}
\label{sec:homotopy}

One other topic within the realm of topology whose study is of great interest is that of homotopies. In order to motivate the study of homotopies from an algebraic perspective, consider the following example, inspired from the explanation given by Nick Alger \autocite{mathstack}.

Suppose we have a malleable topological object $A$ in a topological space\footnote{For the purposes of this example, a topological space $B$ simply is the environment where the object $A$ lives, without going into further details nor technicalities.} $B$. To make this image more concrete, think of the object $A$ as a dough of Play-Doh, which can take many different configurations in space. Now, in order to fully define the particular configuration of the dough of Play-Doh in space, we would need to know the position of all the points of $A$ in $B$. In other words, we need a mapping $f:A \rightarrow B$. Now, if we are to change the Play-Doh's configuration, such as by introducing a dent, we can be even more specific with our data-keeping by defining a map $h:A\times I \rightarrow B$, where $I$ stores all the times $t$ for which this deformation takes place. In particular, we can let $I = [0, T] \subseteq\mathbb{R}$, where $T$ is simply the time when we stop the deformation; then $h(a,t)$, for $a \in A$, $t \in I$, essentially describes the Play-Doh's evolving configuration at a particular snapshot of time $t$. 

Now, the question that we can ask is, for two configurations of the Play-Doh dough $A$ in the topological space $B$, $f:A \rightarrow B$ and $g:A \rightarrow B$, is it possible to continuously transform the configuration defined by $f$ into that defined by $g$, through the time snapshots described by $h(a,t)$? This question is precisely the same as asking whether there exists a homotopy $h$ between the configurations $f$ and $g$. 

In essence, we say that two continuous functions, $f$ and $g$, between topological spaces are homotopic if we can continuously transform/deform one into the other one.  Now, just as with holes in topologies, it is possible to understand homotopies from an algebraic perspective, also using chain complexes. To do so, consider the following constructions.

Let $(\ch{C}, d)$ be a chain complex of vector spaces over a field $\mathbb{F}$. Then,
recall that $\forall n \in \mathbb{Z}$, we had that $0 \subseteq B_n(C) \subseteq Z_n(C) \subseteq C_n$, where for $d_n: C_n \rightarrow C_{n-1}$, $Z_n(C) = \ker(d_n)$ and $B_n(C) = \im(d_{n+1})$. Looking individually at each one of the $n^{th}$ positions of the chain complex, we can construct the following two exact sequences:
\[0 \rightarrow Z_n \rightarrow C_n \rightarrow C_n/Z_n \rightarrow 0\]
\[0 \rightarrow d_{n+1}(C_{n+1}) = B_n \rightarrow Z_n \rightarrow Z_n/d_{n+1}(C_{n+1}) = Z_n/B_n \rightarrow 0\]

Here, we can understand $Z_n$ and $B_n$ as vector sub-spaces of $C_n$ and $Z_n$ respectively. Thus, if we were to define $B_n'(C_n) = C_n/Z_n = d_n(C_n) = B_{n-1}$ and $H_n'(C) = Z_n/B_n = H_n(C)$, then we can create the following two compositions of vector spaces:
\[C_n \cong Z_n \bigoplus C_n/Z_n = Z_n \bigoplus B_n' \] and
\[Z_n \cong B_n \bigoplus Z_n/B_n = B_n \bigoplus H_n' \]

While these structures might seem redundant, the point of them is that they provide a mechanism through which we can move along the chain complex. In particular, they allow us to define an even more important map, which is the splitting map, as can be seen below:

\begin{definition}
A \textbf{splitting map} $s_n: C_n \rightarrow C_{n+1}$ is a map such that we obtain the composition $d_n = d_n \circ s_{n-1} \circ d_n$. Now, a chain complex $C$ is \textbf{split} if we have these maps $s_n$ such that $d_n = d_n \circ s_{n-1} \circ d_n$ for all $n$. The chain complex $C$ is \textbf{split exact} if on top of the existence of the maps $s_n$, we also have that $C$ is an exact sequence (or equivalently, if it is acyclic).
\end{definition}

Notice that this definition makes sense, given that for $d_n: C_n \rightarrow C_{n-1}$, the composition $d_n \circ s_{n-1} \circ d_n$ is essentially performing the following series of mappings $C_n \overset{d_n}\rightarrow C_{n-1} \overset{s_{n-1}}\rightarrow C_n \overset{d_n}\rightarrow C_{n-1}$. In particular, $d_{n+1} \circ s_n$ and $s_{n-1} \circ d_n$ are projection maps from $C_n$ into $B_n$ and $B_n'$ respectively. As a result, the direct sum $d_{n+1} \circ s_n + s_{n-1} \circ d_n$ is an endomorphism of $C_n$, and one can check that the kernel of this map is precisely $H_n'$ or $H_n(C)$. Thus, another way of checking whether the chain complex $C$ is exact is by looking at whether $d_{n+1} \circ s_n + s_{n-1} \circ d_n$ is the identity map on the $C_n$. Overall, the reason as to why the exact sequence $C$ of modules (or vector spaces) splits over a field $\mathbb{F}$ is related to the fact that every module (or vector space) over $\mathbb{F}$ is free, and thus projective. This concept will be explored in greater detail in the next section of this work.

Now that we have provided the definitions from above, it is possible to understand homotopies from a homological algebra perspective. To do so, we will consider two chain complexes, $(\ch{C}, d_C)$ and $(\ch{D}, d_D)$, which as we know from the study of homologies, also describe topological spaces. Then, we can introduce the maps $s_n: C_n \rightarrow D_{n+1}$ and maps $f_n: C_n \rightarrow D_n$ between these chain complexes, such that the following diagram results:

\[\begin{tikzcd}
... \arrow[r]
& C_{n+1} \arrow[r, "d_{C,n+1}"] \arrow[d, "f_{n+1}"'] 
& C_n \arrow[r, "d_{C,n}"] \arrow[d, "f_n"] \arrow[dl, "s_n"]
& C_{n-1} \arrow[r] \arrow[d, "f_{n-1}"] \arrow[dl, "s_{n-1}"]
& ...
\\
... \arrow[r]
& D_{n+1} \arrow[r, "d_{D,n+1}"'] 
& D_n \arrow[r, "d_{D,n}"'] 
& D_{n-1} \arrow[r]
& ...
\end{tikzcd}\]

Now, consider the following important definition:

\begin{definition}
A chain map $f$ between the chain complexes $(\ch{C}, d_C)$ and $(\ch{D}, d_D)$ is \textbf{null homotopic} if there exists these maps $s_n: C_n \rightarrow D_{n+1}$ such that we get $f_n = d_{D,n+1} \circ s_n + s_{n-1} \circ d_{C,n}$, $\forall n \in \mathbb{Z}$, and the diagram commutes. Then we call the collection of maps $\{s_n\}$ the chain contraction of $f$. 
\end{definition}

With this, we are ready to define what it means to have a homotopy between topological spaces from a homological algebra perspective, or a homotopy between maps of chain complexes:

\begin{definition}
Any two chain maps $f$ and $g$ from the chain complex $(\ch{C},d_C)$ to $(\ch{D},d_D)$ are \textbf{chain homotopic} if the difference $f-g$ is null homotopic. That is, $f_n-g_n = d_{D,n+1} \circ s_n + s_{n-1} \circ d_{C,n}$. We denote the collection of maps $\{s_n\}$ the chain homotopy from $f$ to $g$. 
Finally, we say that $f: C \rightarrow D$ is a chain homotopy equivalence (otherwise known as a homotopism of chains) if there exists another chain complex map $q: D \rightarrow C$ such that $q \circ f$ and $f \circ q$ are the identity maps of $C$ and $D$ respectively. We write $C \simeq D$ to denote the homotopism.
\end{definition}

Notice that one can check that if $g$ is any chain complex map between $(\ch{C},d_C)$ and $(\ch{D},d_D)$, so will $g + d_{D,n+1} \circ s_n + s_{n-1} \circ d_{C,n}$, at every $n$. Finally, to conclude the discussion of chain homotopies, notice that we can re-write the diagram from above as follows by replacing the chain complex $D$ for $C$:

\[\begin{tikzcd}
... \arrow[r]
& C_{n+1} \arrow[r, "d_{C,n+1}"] \arrow[d, "f_{n+1}"'] 
& C_n \arrow[r, "d_{C,n}"] \arrow[d, "f_n"] \arrow[dl, "s_n"]
& C_{n-1} \arrow[r] \arrow[d, "f_{n-1}"] \arrow[dl, "s_{n-1}"]
& ...
\\
... \arrow[r]
& C_{C,n+1} \arrow[r, "d_{n+1}"'] 
& C_{C,n} \arrow[r, "d_n"'] 
& C_{n-1} \arrow[r]
& ...
\end{tikzcd}\]

 What is interesting about this diagram is that it tells us directly that a chain complex $C$ is split exact if the identity map on $C$ is null homotopic. In this case, the identity map would be given by the $f_n$, which is $d_{C,n+1} \circ s_n + s_{n-1} \circ d_{C,n}$. 
 
 Overall, this section presented the necessary machinery to understand homotopies of two functions between topological spaces from an algebraic perspective, through the use of chain complexes and splitting maps. In the following two sections, the focus will shift towards the mathematics behind the intrinsic beauty of homological algebra.

\section{Projective Resolutions}
\label{sec:kevin}
One very useful concept in homological algebra is that of resolutions. In essence, these are one specific type of the exact sequences encountered earlier, only that their elements are objects from an Abelian category. The importance of resolutions lies in the idea that they can be used to describe the structure of the objects from the Abelian category in question, in particular through the invariants native to these objects. In general terms, invariants denote those mathematical properties of objects that don't see a modification when applying an operation or transformation. For instance, the magnitude of a complex number remains unchanged regardless of whether we take the conjugate or not. 

To formalize our background on resolutions and to illustrate their use in describing structures, throughout this section, we focus our attention on the category of unital $R$-modules (modules that admit the multiplicative identity). And to start our discussion, we first recall the definition of free $R$-modules. 
\begin{definition}
Let $R$ be a ring; let $I$ be a set. The \textbf{free $R$-module on I} is $\bigoplus_{i\in I}R$, where addition and $R$-action are defined component wise. We denote this module as $R[I]$. If $I$ is empty, we set $R[I]=\{0\}$. 
\end{definition}
Embed $I$ in $R[I]$ with the map $\iota:I\rightarrow R[I]$ sending \[i\mapsto (x_j)_{j\in I}\] where $x_j=0$ for $j\neq i$, and $x_j=1$ when $j=i$. This is verifiably an injective map. One can then see that each element in $R[I]$, $(r_i)_{i\in I}$ , can be written $\sum_{i\in I}r_i\iota(i)$, a well defined sum since elements of direct sums are all but finitely zero. It is immediate from this perspective that the $\iota(i)$ are generators for $R[I]$. 

Now, it is possible to define an important quality of free $R$-modules, as given by the theorem below:

\begin{theorem}
(Universal property of free modules). For any $R$-module $M$ and map $f:I\rightarrow M$, there exists a unique homomorphism $\Tilde{f}:R[I]\rightarrow M$ such that $\Tilde{f}\circ \iota=f$.
\end{theorem}
\begin{proof}
Suppose $f$ is a map from $I$ to $M$. Define $\tilde{f}:R[I]\rightarrow M$ by \[\sum_{i\in I}r_i\iota(i)\mapsto \sum_{i\in I}r_if(i)\]
One sees that $\tilde{f}\circ \iota=f$; we leave it to the reader to verify that this is a homomorphism of $R$-modules, and that any other such homomorphism must be equivalent. 
\end{proof}
This theorem implies that one can define maps from a free module to any arbitrary $R$-module $M$ quite easily; simply assign its generators to \emph{any} elements of $M$, for suppose that $M$ is an arbitrary $R$-module, and each $\iota(i)$ is assigned to $m_i\in M$. Then, we have a set map $f:I\rightarrow M$ defined by $i\mapsto m_i$, and the universal property of free modules (Theorem 5.1) gives us a unique homomorphism $\tilde{f}:R[I]\rightarrow M$. This is a useful quality indeed. 
\begin{definition}
An $R$-module $P$ is said to be \textbf{projective} if for any surjective homomorphism $\varphi:M\rightarrow N$ between $R$-modules $M$ and $N$, and any homomorphism $f:P\rightarrow N$, there exists a lifting map $\tilde{f}:P\rightarrow M$ such that $\varphi\circ \tilde{f}=f$, i.e.,\[\begin{tikzcd}
	& M \\
	P & N
	\arrow["f", from=2-1, to=2-2]
	\arrow["\varphi"', two heads, from=1-2, to=2-2]
	\arrow["{\tilde{f}}", dashed, from=2-1, to=1-2]
\end{tikzcd}\]
\end{definition}

Having defined the notion of module projectiveness, we can introduce the proposition below, which highlights why studying free $R$-modules is of importance, especially in the context of resolutions:

\begin{proposition}
Every free module is projective. 
\end{proposition}
\begin{proof}
Let $R[I]$ be the free module on a set $I$, let $\varphi:M\rightarrow N$ be an epimorphism\footnote{Epimorphisms can be understood as the category theory abstraction of surjective maps. Formally written, from a categorical perspective, an epimorphism $\varphi$ is a morphism between objects $A$ and $B$ of category $X$ such that for $\forall C \in X$ and $\psi_1$, $\psi_2$ from $B$ to $C$, we have that $\psi_1 \circ \varphi$ = $\psi_1 \circ \varphi \implies \psi_1$ = $\psi_2$.}, and let $f:R[I]\rightarrow N$ be a homomorphism. For any $i\in I$, we have $f(\iota(i))\in N$. Since $\varphi$ is surjective, there exists an $m_i\in M$ such that $\varphi(m_i)=f(\iota(i))$. Define $g:I\rightarrow M$ by \[
i\mapsto m_i
\]By the universal property of free modules, there exists a unique homomorphism $\tilde{f}:R[I]\rightarrow M$ such that $\tilde{f}\circ \iota =g$. It is clear from construction that $\varphi\circ g=f\circ \iota$. Moreover, for any $\sum_{i\in I}r_i\iota(i)\in R[I]$, \begin{align*}
    \varphi(\tilde{f}\bigg(\sum_{i\in I}r_i\iota(i))\bigg)&=\varphi\bigg(\sum_{i\in I}r_i\tilde{f}(\iota(i))\bigg)\\
    &=\varphi\bigg(\sum_{i\in I}r_ig(i)\bigg)\\
    &=\sum_{i\in I}r_i\varphi(g(i))\\
    &=\sum_{i\in I}r_if(\iota(i))\\
    &=f\bigg(\sum_{i\in I}r_i\iota(i)\bigg)
\end{align*}Thus, $R[I]$ is projective. 
\end{proof}
Indeed, all free modules are projective, but notice the converse is not necessarily true. The following is an example of a non-free projective module. 
\begin{example}
Let $R=\mathbb{Z}/6\mathbb{Z}$. Then, $\mathbb{Z}/3\mathbb{Z}$ is a $\mathbb{Z}/6\mathbb{Z}$ module as $3\mid 6$; it is not free, but it is projective. To prove this, we use the fact that $\mathbb{Z}/6\mathbb{Z}\cong\mathbb{Z}/2\mathbb{Z}\oplus \mathbb{Z}/3\mathbb{Z}$. Let $\iota:\mathbb{Z}/3\mathbb{Z}\rightarrow \mathbb{Z}/2\mathbb{Z}\oplus \mathbb{Z}/3\mathbb{Z}$ be the embedding into the direct sum, and let $\pi:\mathbb{Z}/2\mathbb{Z}\oplus \mathbb{Z}/3\mathbb{Z}\rightarrow\mathbb{Z}/3\mathbb{Z}$ be the canonical projection. If $\varphi:M\rightarrow N$ is a surjective homomorphism and $f:\mathbb{Z}/3\mathbb{Z}\rightarrow N$ a homomorphism, then we have the following diagram. \[\begin{tikzcd}
	&& M \\
	{\mathbb{Z}/2\mathbb{Z}\oplus \mathbb{Z}/3\mathbb{Z}} & {\mathbb{Z}/3\mathbb{Z}} & N
	\arrow["f", from=2-2, to=2-3]
	\arrow["\varphi"',two heads ,from=1-3, to=2-3]
	\arrow["\pi", shift left=1, from=2-1, to=2-2]
	\arrow["\iota", shift left=1, from=2-2, to=2-1]
\end{tikzcd}\]
Since $\mathbb{Z}/6\mathbb{Z}$ is free as a $\mathbb{Z}/6\mathbb{Z}$-module, Proposition 5.1 implies that it is projective as well. Therefore, there exists a lifting map $g:\mathbb{Z}/2\mathbb{Z}\oplus \mathbb{Z}/3\mathbb{Z}\rightarrow M$ such that $\varphi\circ g=f\circ \pi$.
Define the homomorphism $\tilde{f}:\mathbb{Z}/3\mathbb{Z}\rightarrow M$ by $\tilde{f}=g\circ \iota$. Then, \begin{align*}
    \varphi\circ \tilde{f}&=\varphi\circ g\circ \iota\\
    &=f\circ \pi \circ \iota \\
    &=f
\end{align*}
Consequently, $\mathbb{Z}/3\mathbb{Z}$ is projective as a $\mathbb{Z}/6\mathbb{Z}$-module.
\end{example}
The proof of projectiveness in the above example hints at a more general relationship. Indeed, all direct summands of a free module are in fact projective, and the converse is true as well. However, some theory is needed in order to establish this correspondence. We proceed by first recalling Definition 4.1 in the context of short exact sequences, to then introduce two powerful propositions.
\begin{definition}
A short exact sequence \[
0\rightarrow A\overset{f}{\rightarrow}B\overset{g}{\rightarrow}C\rightarrow 0
\] is said to be \textbf{split exact} if there exists an $h:C\rightarrow B$ such that $gh=id_C$.
\end{definition}
\begin{proposition}
If the above sequence is split exact, then $B\cong A\oplus C$.
\end{proposition}
\begin{proof}
Suppose that the sequence is split exact, so that there exists such an $h:C\rightarrow B$. Define $\varphi: A\oplus C\rightarrow B$ by \[(a,c)\mapsto f(a)+h(c),\] a verifiable homomorphism. We see that $\varphi$ is injective, for suppose that $\varphi(a,c)=0$. Then, $f(a)+h(c)=0$, implying that $g(f(a))+g(h(c))=0$. But, $g(h(c))=c$, and $g(f(a))=0$; thus, $c=0$. This means that $f(a)=0$, but since $f$ is injective, $a=0$. To prove surjectivity, let $b\in B$. Then, $b=(b-h(g(b)))+h(g(b))$. It follows that $g(b-h(g(b)))=g(b)-g(b)=0$, that is, $(b-h(g(b)))\in \ker(g)=\im(f)$. Hence, there exists an $a_0\in A$ such that $f(a_0)=(b-h(g(b)))$. One can then see that $\varphi(a_0,g(b))=b$, as desired. 
\end{proof}
\begin{proposition}
Every $R$-module is the homomorphic image of a free module. Hence, every $R$-module is the homomorphic image of a projective module.
\end{proposition}
\begin{proof}
Let $A$ be the set of generators for $M$, an arbitrary $R$-module. $A$ exists, as we could let $A=M$. Consider the free module on $A$, $R[A]$; the universal property of free modules induces a map $\tilde{f}:R[A]\rightarrow M$, since we have an embedding $id_A:A\rightarrow M$ as well as $\iota:A\rightarrow R[A]$. It is necessarily true that $\tilde{f}\circ \iota=id_A$, so $A\subseteq \im(\tilde{f})$. Since $A$ generates $M$, we can write for any arbitrary $m\in M$, \begin{align*}
    m&=r_1a_1+...r_na_n\\
    &=r_1\tilde{f}(x_1)+...+r_n\tilde{f}(x_n)\\
    &=\tilde{f}(r_1 x_1+...+r_n x_n)
\end{align*}
\end{proof}

Having established these propositions, it is now possible to construct the following theorem, which greatly facilitates the task of finding projective modules.

\begin{theorem}
The following conditions are equivalent for an $R$-module $P$. \begin{enumerate}[(i)]
    \item $P$ is projective. 
    \item Every exact sequence \[
    0\rightarrow A\overset{f}{\rightarrow}B\overset{g}{\rightarrow}P\rightarrow 0
    \] is split exact.
    \item $P$ is a direct summand of a free module. 
\end{enumerate}
\end{theorem}
\begin{proof}
$(i)\implies (ii)$ The map $g:B\rightarrow P$ is an epimorphism (surjective homomorphism), and we have the identity map $P\rightarrow P$. Thus, there exists an $h:P\rightarrow B$ such that $g\circ h=id_P$.

$(ii)\implies (iii)$ Let $A$ be a set of generators for $P$. We have by Proposition 5.3 the exact sequence \[
0\rightarrow \ker(g)\overset{id}{\rightarrow}R[A]\overset{g}{\rightarrow}P\rightarrow 0
\] implying by Proposition 5.2 that $P\cong \ker(g)\oplus R[A]$.

$(iii)\implies (i)$ This proof is structurally identical to the one given in Example 5.1. 
\end{proof}
To illustrate the power of this theorem in determining whether a module is projective or not, consider the following two examples.
\begin{example}
$\mathbb{Z}/2\mathbb{Z}$ is a non-free projective $\mathbb{Z}/6\mathbb{Z}$-module, as it is a direct summand of a free $\mathbb{Z}/6\mathbb{Z}$-module, $\mathbb{Z}/2\mathbb{Z}\oplus \mathbb{Z}/3\mathbb{Z}$.
\end{example}
\begin{example}
Let $R=\mathbb{Z}/2\mathbb{Z}\oplus \mathbb{Z}/2\mathbb{Z}$, where multiplication is defined component wise. Let $\mathbb{Z}/2\mathbb{Z}$ be an $R$-module where the $R$-action is defined \[
(r_1,r_2)\cdot x = r_1x 
\]Then $\mathbb{Z}/2\mathbb{Z}$ is a non-free projective $R$-module.
\end{example}
Now that we are acquainted with projective modules, we are finally ready to introduce the formal framework behind resolutions, which as we mentioned at the beginning of this section, are useful in defining the structure of objects from Abelian categories, such as the category of unital $R$-modules. In this particular case, given an $R$-module $M$, we are interested in examining how closely it resembles a free module. By Proposition 5.3, we know that there exists a free module $P_0$ and epimorphism $\varphi_0:P_0\rightarrow M$, and Theorem 1.1 (First Isomorphism Theorem) asserts that $P_0/\ker(\varphi_0)\cong M$. If $\ker(\varphi_0)=0$, then we are satisfied; $M$ is itself a free module. If not, then $M$ need not be free, but can certainly be extended by $\ker(\varphi_0)$ to a free module. We can measure how free this extension is as well, for $\ker(\varphi_0)$ likewise need not be free, so choose a free module $P_1$ and epimorphism $\varphi_1:P_1\rightarrow \ker(\varphi_0)$. Iterating this process infinitely many times, we can construct a chain complex of free modules that maps onto $M$. Proposition 5.1 asserts that we could have done this same procedure with projective modules. Formally, we can make the following definition.
\begin{definition}
Let $M$ be an $R$-module. A \textbf{projective resolution of M}, denoted $P$, is a collection of projective modules $\{P_i\}_{i\in \mathbb{N}}$ and homomorphisms $\varphi_i:P_{i}\rightarrow P_{i-1}$ where $P_{-1}=M$ such that the following chain complex is exact. \[
...\overset{\varphi_3}{\rightarrow} P_2\overset{\varphi_2}{\rightarrow}P_1\overset{\varphi_1}{\rightarrow}P_0\overset{\varphi_0}{\rightarrow}M\rightarrow 0
\]
If each $P_i$ is free, we call it a $\textbf{free resolution}$.
\end{definition}
The discussion preceding Definition 5.4 shows that any $R$-module $M$ has a free resolution, defined inductively for each degree by choosing a free module mapping onto the kernel of the prior degree homomorphism. Of course, we can say the same thing about projective resolutions, as Proposition 5.1 implies that a free resolution is a projective resolution. For subsequent results, we only require that the resolution be projective. If for some $i\in \mathbb{N}$ we have that $\varphi_i$ is bijective, we say that the resolution is finite. Consider the following example of a finite projective resolution. 
\begin{example}
Let $R=\mathbb{Z}$; let $M=\mathbb{Z}/2\mathbb{Z}$.
\[
...\rightarrow0\rightarrow \mathbb{Z}\overset{\varphi_1}{\rightarrow}\mathbb{Z}\overset{\varphi_0}{\rightarrow}\mathbb{Z}/2\mathbb{Z}\rightarrow 0
\]The above sequence where $\varphi_0$ maps $x\mapsto x+2\mathbb{Z}$, i.e., it is the quotient homomorphism, and $\varphi_1$ maps $x\mapsto 2x$ is an exact chain complex, and since we have a bijection with the mapping $x\mapsto 2x$, it is a finite projective resolution. One may assume that every $P_i$ for $i>1$ is $0$.
\end{example}
Here's an infinite projective resolution. 
\begin{example}
Let $R=\mathbb{Z}/4\mathbb{Z}$; let $M=\mathbb{Z}/2\mathbb{Z}$.\[
...\overset{\varphi_2}{\rightarrow} \mathbb{Z}/4\mathbb{Z}\overset{\varphi_1}{\rightarrow}\mathbb{Z}/4\mathbb{Z}\overset{\varphi_0}{\rightarrow}\mathbb{Z}/2\mathbb{Z}\rightarrow 0
\]The above sequence where $\varphi_0$ maps $x+4\mathbb{Z}\mapsto x+2\mathbb{Z}$, and $\varphi_i$ for $i\geq1$ maps $x+4\mathbb{Z}\mapsto 2x+4\mathbb{Z}$ is an exact chain complex, and since none of the $\varphi_i$ are bijective, it is an infinite projective resolution.
\end{example}
Now, going back to the original motivation for studying resolutions, notice that for a projective resolution of $M$, we may remove $M$ from the sequence without loss of data since $M\cong \coker(\varphi_1)$. Effectively, this shows that a projective resolution of $M$ completely encodes the structure of $M$, and thus one can present $M$ by its projective resolution. Naturally, one might ask if this presentation is unique. In a way, it is. This section's final result shows that any projective resolution of a given $R$-module is unique up to quasi-isomorphism. We present it, prove it, and conclude this section. 
\begin{theorem}
Suppose $P$ and $P'$ are projective resolutions of $M$, as shown in the diagram below. \[\begin{tikzcd}
	{...} & {P_2} & {P_1} & {P_0} & M & 0 \\
	{...} & {P_2'} & {P_1'} & {P_0'} & M & 0
	\arrow["{\varphi_3}", from=1-1, to=1-2]
	\arrow["{\varphi_2}", from=1-2, to=1-3]
	\arrow["{\varphi_1}", from=1-3, to=1-4]
	\arrow["{\varphi_0}", from=1-4, to=1-5]
	\arrow["{\psi_3}", from=2-1, to=2-2]
	\arrow["{\psi_2}", from=2-2, to=2-3]
	\arrow["{\psi_1}", from=2-3, to=2-4]
	\arrow["{\psi_0}", from=2-4, to=2-5]
	\arrow[from=1-5, to=1-6]
	\arrow[from=2-5, to=2-6]
\end{tikzcd}\]
Then, there exists a chain map $f=\{f_i:P_i\rightarrow P_i'\}_{i\in \mathbb{N}}$ that induces an isomorphism on homologies $\tilde{f}_i:H_i(P)\rightarrow H_i(P')$ for all $i\in \mathbb{N}$.
\end{theorem}
\begin{proof}
We define $f$ inductively. By hypothesis, $\psi_0:P_0'\rightarrow M$ is a surjection, and since $P_0$ is projective, we have a lifting $f_0:P_0\rightarrow P_0'$ such that $\psi_0\circ f_0=\varphi_0$.
\[\begin{tikzcd}
	& {P_0'} \\
	{P_0} & M
	\arrow["{\varphi_0}", from=2-1, to=2-2]
	\arrow["{\psi_0}"', two heads, from=1-2, to=2-2]
	\arrow["{f_0}", dashed, from=2-1, to=1-2]
\end{tikzcd}\]
Of course, $id_M\circ \varphi_0=\psi_0\circ f_0$, so for the purposes of our induction, we say that $id_M=f_{-1}$.
\[\begin{tikzcd}
	{...} & {P_2} & {P_1} & {P_0} & M & 0 \\
	{...} & {P_2'} & {P_1'} & {P_0'} & M & 0
	\arrow["{\varphi_3}", from=1-1, to=1-2]
	\arrow["{\varphi_2}", from=1-2, to=1-3]
	\arrow["{\varphi_1}", from=1-3, to=1-4]
	\arrow["{\varphi_0}", from=1-4, to=1-5]
	\arrow["{\psi_3}", from=2-1, to=2-2]
	\arrow["{\psi_2}", from=2-2, to=2-3]
	\arrow["{\psi_1}", from=2-3, to=2-4]
	\arrow["{\psi_0}", from=2-4, to=2-5]
	\arrow[from=1-5, to=1-6]
	\arrow[from=2-5, to=2-6]
	\arrow["{f_0}"', dashed, from=1-4, to=2-4]
	\arrow["id"', from=1-5, to=2-5]
\end{tikzcd}\]

Now, suppose that $f_n:P_n\rightarrow P_n'$ and $f_{n-1}:P_{n-1}\rightarrow P_{n-1}'$ exist such that $\psi_n\circ f_n=f_{n-1}\circ\varphi_n$. For any $x\in P_{n+1}$, \begin{align*}
    \psi_n\circ (f_n\circ \varphi_{n+1})(x)&=(\psi_n\circ f_n)\circ \varphi_{n+1}(x)\\
    &=(f_{n-1}\circ \varphi_n)\circ \varphi_{n+1}(x)\\
    &=f_{n-1}\circ (\varphi_n\circ \varphi_{n+1})(x)\\
    &=0
    \end{align*}
By Proposition 2.1, $\im(f_n\circ\varphi_{n+1})\subseteq \ker(\psi_n)$. Since $\psi_{n+1}$ surjects onto $\ker(\psi_n)$, $P_{n+1}$ projective implies that there exists a $f_{n+1}:P_{n+1}\rightarrow P_{n+1}'$ such that $\psi_{n+1}\circ f_{n+1}=f_{n}\circ \varphi_{n+1}$. This gives us $f_{i}$ for all $i\in \mathbb{N}$. 

Define $\tilde{f_i}:H_i(P)\rightarrow H_i(P')$ by \[
k+\im(\varphi_{i+1})\mapsto f_i(k)+\im(\psi_{i+1})
\]We leave it to the reader to verify that this is a well defined homomorphism. Observe that for any $i\in \mathbb{N}$, $H_i(P)=\ker(\varphi_i)/\im(\varphi_{i+1})=\{0+\im(\varphi_{n+1})\}$ and $H_i(P')=\ker(\psi_i)/\im(\psi_{i+1})=\{0+\im(\psi_{i+1})\}$. As a homomorphism, $\tilde{f_i}$ maps $0+\im(\varphi_{n+1}\mapsto 0+\im(\psi_{i+1})$; this is trivially onto and one-to-one. Consequently, we can see that $f$ induces an isomorphism between $H_i(P)$ and $H_i(P')$ for each $i\in \mathbb{N}$.   
\end{proof}

In the last section of this work, we continue exploring the intrinsic abstract mathematics behind homological algebra and examine yet a further application of projective resolutions, in the context of category theory.

\section{Tor Functor}
\label{sec:robin}
As was mentioned in the previous section, projective resolutions are useful in the sense that they help describe the structure of objects in an Abelian category through their invariants. However, this is related to yet one powerful application. One of the central theorems in homological algebra is the ``universal coefficient theorem", a key component of algebraic topology that seeks to relate homology groups of different topological shapes to those of simplicial complexes (multi-dimensional triangles). Conceptually, this is the algebraic framework through which we can simplify our analysis and classification of topologies, by looking at simplicial complexes instead of complicated shapes. Tor functors, the objects of discussion in this section, are the mechanism through which this association is done.

Before continuing, it is important to refamiliarize ourselves with an essential concept of module theory. Recall that given $R$-modules $M$ and $N$, we may define the tensor product $M \otimes_R N$ as an $R$-module, equipped with a map $\iota: M \times N \rightarrow M \otimes_R N$. Recall the tensor product satisfies the condition that any $R$-bilinear map $f: M \times N \rightarrow S$ induces an $R$-linear map $\overline{f}: M \otimes_R N \rightarrow S$ such that $f = \overline{f} \circ \iota$. The proofs of existence and uniqueness of the tensor product are given in most abstract algebra books, including Dummit and Foote \autocite{dummitandfoote}: for conciseness, we will not repeat them here.

We now introduce briefly another concept from category theory to develop the language needed to progress with homological algebra. Functors are one of the main building blocks of category theory, acting as the bridges between categories. Formally defined, for two categories $C$ and $D$, a functor $F$ will map all objects $A \in C$ to respective objects $F(A)\in D$ and  will map all morphisms $(f:A\rightarrow B) \in C$ to respective morphisms $(F(f):F(A)\rightarrow F(B)) \in D$. Functors must preserve identity morphisms, so that $F(id_X) = id_{F(X)}$, for all $X \in C$, and functors must preserve composition of maps, so $F(f \circ g) = F(f) \circ F(g)$, for all $f:X \rightarrow Y$ and $g:Y \rightarrow Z \in C$. As can be seen, this is an additional level of abstraction than the simpler set or group theory, and it can continue growing, by considering the category of categories and the functors of functors.

Note that by definition, functors preserve composition of maps, so they also preserve commutative diagrams. Now, our goal in this section is to use projective resolutions to construct something known as the left-derived functors of an $R$-module. However, before explaining what left-derived functors are, we begin by defining what is known as a right-exact functor.
\begin{definition}
A functor $F: \rmod{R} \rightarrow \rmod{S}$ between modules in $R$ and modules in $S$ is $\textbf{right-exact}$ if for every short exact sequence of $R$-modules \[
0 \rightarrow A\rightarrow B\rightarrow C\rightarrow 0 
\] it follows that the sequence \[
F(A) \rightarrow F(B)\rightarrow F(C)\rightarrow 0 
\] is exact in $S$ modules, but the map $F(A) \rightarrow F(B)$ may not be injective.
\end{definition}
The notion of a $\textbf{left-exact}$ functor can be defined dually; alternatively, we can say right-exact functors preserve cokernels and left-exact functors preserve kernels. These definitions can be extended with ease to any Abelian category, but for this paper we are only concerned with these functors over $R$-modules. An example of a right-exact functor is the tensor product $M \otimes_{R} -:Mod_{R} \rightarrow Mod_{R}$, for $R$ commutative. Note that this functor takes each $R$-module $N$ to the $R$-module $M \otimes_R N$, and each $R$-module homomorphism $f: N \rightarrow P$ to the morphism $M \otimes_R f: M \otimes_R N \rightarrow M \otimes_R P$ defined by $(M \otimes_R f) (m \otimes n) = m \otimes f(n)$.
Now, given a right exact functor $F$ and a $R$-module $M$, consider any projective resolution of $M$:  \[
\ldots \rightarrow P_2\rightarrow P_1\rightarrow P_0\rightarrow M 
\] Then apply $F$ to the projective resolution: \[
\ldots \rightarrow F(P_2) \rightarrow F(P_1)\rightarrow F(P_0) 
\] Note that $\coker(F(P_1) \rightarrow F(P_0)) \cong F(M)$. Computing the homology of this sequence at the $i$-th coordinate gives an object, denoted: \[L_iF(M) = \ker(F(P_i) \rightarrow F(P_{i-1}))/ \im(F(P_{i+1}) \rightarrow F(P_i))\]
We further define $L_0F(M)=F(M)$. Furthermore, consider any function $M \overset{f}{\rightarrow} N$, and choose projective resolutions $P_i$ and $Q_i$ for $M$ and $N$ respectively.
\[\begin{tikzcd}
	\ldots & P_2 & P_1 & P_0 & M & 0\\
	\ldots & Q_2 & Q_1 & Q_0 & N & 0
	\arrow[from=1-1, to=1-2]
	\arrow[from=1-2, to=1-3]
	\arrow[from=1-3, to=1-4]
	\arrow[from=1-4, to=1-5]
	\arrow[from=1-5, to=1-6]
	\arrow[from=2-1, to=2-2]
	\arrow[from=2-2, to=2-3]
	\arrow[from=2-3, to=2-4]
	\arrow[from=2-4, to=2-5]
	\arrow[from=2-5, to=2-6]
	\arrow["f", from=1-5, to=2-5]
\end{tikzcd}\]
Then consider the composite $P_0 \rightarrow M \overset{f}{\rightarrow} N$. Then since $Q_0$ is projective and the map $Q_0 \rightarrow N$ is surjective, it follows that there exists an induced map $f_0$ such that the diagram commutes:
\[\begin{tikzcd}
	\ldots & P_2 & P_1 & P_0 & M & 0\\
	\ldots & Q_2 & Q_1 & Q_0 & N & 0
	\arrow[from=1-1, to=1-2]
	\arrow[from=1-2, to=1-3]
	\arrow[from=1-3, to=1-4]
	\arrow[from=1-4, to=1-5]
	\arrow[from=1-5, to=1-6]
	\arrow[from=2-1, to=2-2]
	\arrow[from=2-2, to=2-3]
	\arrow[from=2-3, to=2-4]
	\arrow[from=2-4, to=2-5]
	\arrow[from=2-5, to=2-6]
	\arrow["f_0", dashed, from=1-4, to=2-4]
	\arrow["f", from=1-5, to=2-5]
\end{tikzcd}\]
To construct $f_1$, note in particular since the diagram commutes we have that the composite \[
(P_1 \rightarrow P_0 \rightarrow Q_0 \rightarrow N) = (P_1 \rightarrow P_0 \rightarrow M) \rightarrow N = f \circ 0 = 0.
\] So in particular: \[
\im(P_1 \rightarrow P_0 \rightarrow Q_0) \subseteq \ker(Q_0 \rightarrow N) = \im(Q_1 \rightarrow Q_0)
\]Hence $f_0$ restricts onto the image of $Q_1 \rightarrow Q_0$, and then since $Q_1$ is projective and the map $Q_1 \rightarrow \im(Q_1 \rightarrow Q_0)$ is surjective, we do indeed induce a map $f_1$ from $P_1$ to $Q_1$. In the same way, we can define maps $f_2, f_3 \ldots$ such that the diagram commutes:
\[\begin{tikzcd}
	\ldots & P_2 & P_1 & P_0 & M & 0\\
	\ldots & Q_2 & Q_1 & Q_0 & N & 0
	\arrow[from=1-1, to=1-2]
	\arrow[from=1-2, to=1-3]
	\arrow[from=1-3, to=1-4]
	\arrow[from=1-4, to=1-5]
	\arrow[from=1-5, to=1-6]
	\arrow[from=2-1, to=2-2]
	\arrow[from=2-2, to=2-3]
	\arrow[from=2-3, to=2-4]
	\arrow[from=2-4, to=2-5]
	\arrow[from=2-5, to=2-6]
	\arrow["f_2", dashed, from=1-2, to=2-2]
	\arrow["f_1", dashed, from=1-3, to=2-3]
	\arrow["f_0", dashed, from=1-4, to=2-4]
	\arrow["f", from=1-5, to=2-5]
\end{tikzcd}\]
Then applying $F$ maintains that the diagram commutes, so taking the homology induces maps:
\[\begin{tikzcd}
	\ldots & L_2F(M) & L_1F(M) & F(M)\\
	\ldots & L_2F(N) & L_1F(N) & F(N)
	\arrow["0", from=1-1, to=1-2]
	\arrow["0", from=1-2, to=1-3]
	\arrow["0", from=1-3, to=1-4]
	\arrow["0", from=2-1, to=2-2]
	\arrow["0", from=2-2, to=2-3]
	\arrow["0", from=2-3, to=2-4]
	\arrow["F(f_2)", from=1-2, to=2-2]
	\arrow["F(f_1)", from=1-3, to=2-3]
	\arrow["F(f_0)", from=1-4, to=2-4]
\end{tikzcd}\]
We denote $L_iF(f)=F(f_i)$. To summarize, we have constructed, for each $i \in \mathbb{N}$, a set of $R$-modules $L_iF(M)$ and a set of maps $L_iF(f): L_iF(M) \rightarrow L_iF(N)$. We consolidate this information in the following definition:
\begin{definition}
Given an $R$-module $M$ and a right-exact functor $F$, we compute the $\textbf{left-derived functors}$ $L_iF$ as above.
\end{definition}
We leave it to the interested reader to show that $L_iF$ is indeed a functor from $\rmod{R}$ to $\rmod{R}$. Currently, we have defined the left-derived functors based on a specific projective resolution: we will now see that any projective resolution suffices.
\begin{theorem}
$L_iF(M)$ always exists and is independent of the choice of projective resolution. 
\end{theorem}
\begin{proof}
Existence follows immediately from the construction of a free resolution over a module, as discussed preceding Definition 5.4, and the fact that any free resolution is also a projective resolution. A proof of uniqueness can be found at Weibel \S 2.4.1 \autocite{weibel}.        
\end{proof}
\begin{proposition}
If $M$ is projective, then $L_iF(M)=0$ for all $i>0$.
\end{proposition}
\begin{proof}
One particular projective resolution of $M$ is \[
\ldots \rightarrow M\overset{id}\rightarrow M\overset{id}\rightarrow M\overset{id}\rightarrow M 
\]
The functor preserves identity maps, so the result is clear.
\end{proof}
Now, to explain the motivation for left-derived functors, we can start by considering the following theorem.
\begin{theorem}
Given an exact sequence \[
0 \rightarrow A \overset{f}{\rightarrow} B \overset{g}{\rightarrow} C \rightarrow 0 
\]
and a right-exact functor $F$, it follows that the left-derived functors of $F$ induce the following long exact sequence:
\[\begin{tikzcd}
	\ldots & L_2F(A) & L_2F(B) \arrow[d, phantom, ""{coordinate, name=Y}] & L_2F(C)&\\
	& L_1F(A) & L_1F(B) \arrow[d, phantom, ""{coordinate, name=Z}] & L_1F(C)&\\
	& L_0F(A) & L_0F(B) & L_0F(C)&0
	\arrow[from=1-1, to=1-2]
	\arrow["L_2F(f)", from=1-2, to=1-3]
	\arrow["L_2F(g)", from=1-3, to=1-4]
	\arrow[from=1-4, to=2-2,
	rounded corners,
to path={ -- ([xshift=2ex]\tikztostart.east) \tikztonodes
|- (Y) [near end]
-| ([xshift=-2ex]\tikztotarget.west)
-- (\tikztotarget)},
"\delta_2"
	]
	\arrow["L_1F(f)", from=2-2, to=2-3]
	\arrow["L_1F(g)", from=2-3, to=2-4]
	\arrow[from=2-4, to=3-2,
rounded corners,
to path={ -- ([xshift=2ex]\tikztostart.east)\tikztonodes
|- (Z) [near end]
-| ([xshift=-2ex]\tikztotarget.west)
-- (\tikztotarget)},
"\delta_1"
	]
	\arrow["F(f)", from=3-2, to=3-3]
	\arrow["F(g)", from=3-3, to=3-4]
	\arrow[from=3-4, to=3-5]
\end{tikzcd}\]
\end{theorem}
\begin{proof}
The proof is lengthy and technical, although it is based applying ``Snake Lemma" to get the connecting homomorphismss, as were shown above. More details can be found at Rotman \S 6.27 \autocite{rotman}
\end{proof}
However, a perhaps bigger motivation for left-derived functors lies in using them to define the Tor functor, the focus of this section. 
\begin{definition}
The \textbf{Tor functors} are the left-derived functors of $M \otimes_R -$. In particular, we denote $\Tor^{R}_i(M, N)$ as the $i$-th left-derived functor of $M \otimes_R -$, evaluated at $N$.
\end{definition}
The above is isomorphic to the $i$-th left-derived functor of $- \otimes_R N$ evaluated at $M$ in a similar manner to the isomorphism $M \otimes_R N \cong N \otimes_R M$, though this statement is not obvious. Hence we may also consider Tor as a functor of two variables $\Tor^{R}_i(-, -): \rmod{R} \times \rmod{R} \rightarrow \rmod{R}$. When the ring $R$ is understood, we frequently use the notation $\Tor_i(M, N)$ instead of $\Tor^{R}_i(M, N)$.

To illustrate this definition of the Tor functor in action, we can consider the following simple example. 

\begin{example}
Let $R = \mathbb{C}[x,y]$, with the ideal $I = (x,y)$. Then, $\Tor_1(R/I, I) = \Tor_2(R/I, R/I) = \mathbb{C}$.
\end{example}

Notice that the name Tor has to do with a link between the Tor functor and torsion, exemplified in the following exercise:
\begin{proposition}
Given a ring $R$, a $R$-module $M$, and a nonzerodivisor $x \in R$, we have that $\Tor_1(R/(x), M) \cong \{m \in M \mid xm = 0\}$.
\end{proposition}
\begin{proof}
Consider the exact sequence \[
0 \rightarrow R \overset{\cdot x}{\rightarrow} R \rightarrow R/(x) \rightarrow 0 
\]
The map $R \rightarrow R/(x)$ is the canonical projection, and the map $\cdot x$ is just ring multiplication in $R$ by $x$. Because $R$ is free and hence projective, we can consider this exact sequence as a projective resolution of $R/(x)$: \[
\ldots \rightarrow 0 \rightarrow 0 \rightarrow R \overset{\cdot x}{\rightarrow} R \rightarrow R/(x)
\]
We then apply tensor by $M$: \[
\ldots \rightarrow 0 \rightarrow 0 \rightarrow R \otimes_R M \overset{\cdot x \otimes id_M}{\longrightarrow} R \otimes_R M \rightarrow R/(x) \otimes_R M
\]
Hence \[
\Tor_1(R/(x), M) = \ker(\cdot x \otimes id_M) / \im(0 \rightarrow R \otimes_R M) \cong \ker(\cdot x \otimes id_M)
\]
However, we have $R \otimes_R M \cong M$ via the bijection $r \otimes m \mapsto rm$. So $\ker(\cdot x \otimes id_M)$ is exactly those elements $r \otimes_R m$ for which $xrm=0$, which bijects to exactly those elements $rm \in M$ for which $xrm = 0$. Since $x$ is a nonzerodivisor, in this case the Tor functor does compute the torsion.
\end{proof}

Now, naturally from its definition, the Tor functor also provides some very valuable insights into understanding the tensor product operator, as well as some of its applications in Ring theory, as can be seen in the proposition below.  

\begin{proposition}
Given a ring $R$ with ideals $I$, $J$, we have that \[\Tor_1(R/I, R/J) \cong (I \cap J)/IJ\]
\end{proposition}
\begin{proof}
Consider the exact sequence \[
0 \rightarrow J \hookrightarrow R \rightarrow R/J \rightarrow 0 
\]
The maps are of course the inclusion and natural projection. Consider the long exact sequence induced by the Tor functor with $R/I$ (Theorem 6.2):
\[\begin{tikzcd}
	\ldots & \Tor_1(R/I, J) & \Tor_1(R/I, R) \arrow[d, phantom, ""{coordinate, name=Y}] & \Tor_1(R/I, R/J)&\\
	& R/I \otimes_R J & R/I \otimes_R R & R/I \otimes_R R/J&0
	\arrow[from=1-1, to=1-2]
	\arrow[from=1-2, to=1-3]
	\arrow[from=1-3, to=1-4]
	\arrow[from=1-4, to=2-2,
	rounded corners,
to path={ -- ([xshift=2ex]\tikztostart.east) \tikztonodes
|- (Y) [near end]
-| ([xshift=-2ex]\tikztotarget.west)
-- (\tikztotarget)}]
	\arrow[from=2-2, to=2-3]
	\arrow[from=2-3, to=2-4]
	\arrow[from=2-4, to=2-5]
\end{tikzcd}\]
Recall that $R$ is a free and hence projective module. Furthermore, recall that Tor is a left-derived functor, so we can see from Proposition $6.1$ that $\Tor_1(R/I, R) = 0$. We also note that $R/I \otimes_R R \cong R/I$ and $R/I \otimes_R J \cong J/IJ$, so we can rewrite the long exact sequence: \[
\ldots \rightarrow 0 \rightarrow \Tor_1(R/I, R/J) \rightarrow J/IJ \rightarrow R/I \rightarrow R/I \otimes_R R/J \rightarrow 0
\]
In particular we find that since the map $0 \rightarrow \Tor_1(R/I, R/J)$ has a trivial image, the kernel of the map $\Tor_1(R/I, R/J) \rightarrow J/IJ$ is trivial, so in fact we have by the first isomorphism theorem and exactness \[
\Tor_1(R/I, R/J) \cong \im(\Tor_1(R/I, R/J) \rightarrow J/IJ) \cong \ker(J/IJ \rightarrow R/I)
\]
We wish to understand the map $J/IJ \rightarrow R/I$. Reversing the isomorphisms, we find that this is the inclusion map 
\[R/I \otimes_R J \rightarrow R/I \otimes_R R\]
\[r+I \otimes s \mapsto r+I \otimes s\]
So after the isomorphisms, this is simply the map sending $r+IJ$ to $r+J$ for $r \in I$. So  the kernel of this map is exactly those elements $r+IJ$ for which $r \in J$ as well, giving indeed that $\ker(J/IJ \rightarrow R/I) \cong (I \cap J)/IJ$ as desired. 
\end{proof}
We now take a moment to pause and receive the skeptical claims that the definitions of projective resolutions and Tor are not, by nature, very constructible structures, especially when working in categories other than $R$-modules. Although this is true, the above examples and exercises indeed reveal that even in the cases where Tor is constructible, the functors do indeed carry a deep meaning of some of the fundamental properties of objects. For example, the previous proposition immediately implies that if $I+J=R$, then $IJ = I \cap J$, a central result about ideals of a ring. In addition, the Tor functor also has deep connections to flat\footnote{By definition, a module $M$ over a ring $R$ is flat if $\forall f:A\rightarrow B$ injective linear maps of $R$-modules, then the map $f \otimes_R M: A \otimes_R M\rightarrow B \otimes_R M$ is injective too. An interesting result about flatness is that $\forall i>0, \Tor^{R}_i(M, N) = 0$ if either $A$ or $B$ themselves are flat.} $R$-modules, which the interested reader can explore further in Dummit and Foote \S 17.1 \autocite{dummitandfoote}. Yet, as was alluded at the beginning of this section, the biggest application of the Tor functor lies in helping prove the ``universal coefficient theorem" for homology, one of the most important results from homological algebra. Our hope is that, after reading this guide, abstract algebra students will be equipped with the necessary framework to delve deeper into and appreciate the beauty of these fascinating topics.

\section{Conclusion}
Homological algebra, like many parts of abstract algebra, is an exercise in generalizing known ideas to create different ways of thinking about problems; as was evident in this work, the treated topics naturally had deep connections to topology and category theory, from which we motivated examples and applications of this field. 

We started this work by exploring the main building blocks behind homological algebra, most notably chain complexes, homology modules, and homotopies. Throughout this process, we discovered ideas such as being able to construct chain complexes of chain complexes through an application the ``Snake Lemma", as well as new ways of thinking about the classification of topological shapes, through the computation of holes at various dimensions, as enabled by homology modules. We also introduced the notion of what it means for two chain complexes to be chain homotopic, which is the algebraic analog of asking about when can one shape by continuously deformed into another. Having established this foundation, we then moved into more advanced topics, including projective resolutions and Tor functors. In these sections, we highlighted the importance of these two concepts in helping describe algebraic structures through their invariants, but also illustrated why homological algebra is a field worth studying by itself, beyond the topological applications.

Overall, through this guide, we hope to not only have provided our reader with the necessary background to further study homological algebra, but also with enough motivation and excitement to do so.

\newpage
\nocite{weibel}
\nocite{dummitandfoote}
\nocite{eisenbud}
\nocite{hungerford}
\nocite{rotman}
\nocite{gardner}
\nocite{mathstack}
\nocite{quanta}
\printbibliography

\end{document}

%% file: prelude.tex
\DeclareSymbolFont{bbold}{U}{bbold}{m}{n}
\DeclareSymbolFontAlphabet{\mathbbold}{bbold}

\renewcommand{\hom}[1]{\mathrm{Hom}_{#1}}

\DeclareMathOperator{\Ob}{\mathrm{Ob}}
\DeclareMathOperator{\id}{\operatorname{id}}
\DeclareMathOperator{\coker}{\operatorname{coker}}

\DeclareMathOperator{\Ab}{\mathrm{\underline{Ab}}}

\DeclareMathOperator{\injectsto}{\hookrightarrow}

\DeclareMathOperator{\im}{\operatorname{im}}

\DeclareMathOperator{\Tor}{\operatorname{Tor}}

\newcommand{\rmod}[1]{\mathrm{\underline{Mod}}_{#1}}
\newcommand{\ch}[1]{{#1}_{\bullet}}
\newcommand{\Ch}{\mathbf{Ch}}